\documentclass[preprint,12pt]{elsarticle}

\usepackage{amssymb}
\usepackage{amsthm}
\usepackage[caption=false]{subfig}
\newtheorem{theorem}{Theorem}
\newtheorem{lemma}[theorem]{Lemma}

\journal{Elsevier}

\begin{document}

\begin{frontmatter}


\tnotetext[support]{Research supported by the
SEP-CONACYT A1-S-8397 grant of the second author.}
\author[FC]{Juan Carlos Garc\'ia-Altamirano}
\ead{carlos\_treze@ciencias.unam.mx}

\author[FC]{C\'esar Hern\'andez-Cruz\corref{cor1}}
\ead{chc@ciencias.unam.mx}

\address[FC]{Facultad de Ciencias\\
	           Universidad Nacional Aut\'onoma de M\'exico\\
						 Av. Universidad 3000, Circuito Exterior S/N\\
						 Delegaci\'on Coyoac\'an, C.P. 04510\\
						 Ciudad Universitaria, D.F., M\'exico}

\cortext[cor1]{Corresponding Author}

\title{Minimal obstructions for a matrix partition problem
in chordal graphs\tnoteref{support}}


\author{}

\address{}

\begin{abstract}
If $M$ is an $m \times m$ matrix over $\{ 0, 1, \ast \}$,
an $M$-partition of a graph $G$ is a partition $(V_1, \dots
V_m)$ such that $V_i$ is completely adjacent (non-adjacent)
to $V_j$ if $M_{ij} = 1$ ($M_{ij} = 0$), and there are no
further restrictions between $V_i$ and $V_j$ if $M_{ij} =
\ast$.   Having an $M$-partition is a hereditary property,
thus it can be characterized by a set of minimal
obstructions (forbidden induced subgraphs minimal with
the property of not having an $M$-partition).

It is known that for every $3 \times 3$ matrix $M$ over
$\{ 0, 1, \ast \}$, there are finitely many chordal minimal
obstructions for the problem of determining whether a graph
admits an $M$-partition, except for two matrices,
$M_1 = \left(
\begin{array}{ccc}
0    & \ast & \ast \\
\ast & 0    & 1 \\
\ast & 1    & 0
\end{array}
\right)$
and
$M_2 = \left(
\begin{array}{ccc}
0    & \ast & \ast \\
\ast & 0    & 1 \\
\ast & 1    & 1
\end{array}
\right)$.
For these two matrices an infinite family of chordal
minimal obstructions is known (the same family for both
matrices), but the complete set of minimal obstructions
is not.

In this work we present the complete family of chordal
minimal obstructions for $M_1$.
\end{abstract}


\end{frontmatter}


\section{Introduction}

Let $M$ be an $m \times m$ symmetric matrix over $\{ 0 , 1,
\ast \}$, we will call such a matrix a {\em pattern}. An
{\em $M$-partition} of a graph $G$ is a partition $(V_1,
\dots, V_m)$ of $V_G$ such that, for every $i,j \in \{ 1,
\dots, m \}$, every vertex in $V_i$ is adjacent to every
vertex in $V_j$ if $M_{ij} = 1$, every vertex in $V_i$ is
non-adjacent to every vertex in $V_j$ if $M_{ij} = 0$, and
there are no restrictions between vertices of $V_i$ and
vertices of $V_j$ if $M_{ij} = \ast$.   Notice that in
the previous definition we might have $i = j$, in which
case we will have that $V_i$ is a clique if $M_{ii} = 1$,
and $V_i$ is an independent set if $M_{ii} = 0$.   As it
is common in Graph Theory, we do not ask every part in our
partitions to be non-empty, so we will usually forbid the
diagonal entries of our patterns to be $\ast$; if $M_{ii}
= \ast$ for some $i \in \{ 1, \dots, m \}$, then we would
have a part without restrictions, so we might place all the
vertices in $V_i$ to obtain a trivial $M$-partition.   A
graph that admits an $M$-partition is called {\em
$M$-partitionable}.

Many classical problems in Graph Theory can be stated as
$M$-partition problems, e.g., the matrix partition problem
corresponding to a $k \times k$ pattern with all diagonal
entries equal to $0$ and all off-diagonal entries equal to
$\ast$ is just the usual $k$-colouring problem; for every
pattern $M$ without $1$'s, the $M$-partition problem
corresponds to an $H$-homomorphism problem, where $H$ is
the graph whose adjacency matris is obtained from $M$ by
replacing each $\ast$ by a $1$.   The interested reader
may refer to \cite{hellEJC35} for a wonderful survey on
the subject.   Clearly, having an $M$-partition is a
hereditary property, and thus, for a fixed pattern $M$,
the class of $M$-partitionable graphs can be characterized
in terms of a family of forbidden induced subgraphs.   A
graph is a {\em minimal obstruction} for the $M$-partition
problem if it is not $M$-partitionable, but every proper
induced subgraph is.   The set of all minimal obstructions
for an $M$-partition problem is, besides a family of
forbidden induced subgraphs characterizing $M$-partitionable
digraphs, also a set of no-certificates for an algorithm
that verifies if a graph is $M$-partitionable.   Recall
that a {\em certifying algorithm} is an algorithm for a
decision problem which provides a yes-certificate for
every yes-instance of the problem, and a no-certificate
for every no-instance of the problem.   If the validity
of these certificates can be verified faster than the
time it takes to actually solve the problem, then they
provide a valuable tool to check the correctness of an
implementation of the algorithm.

In \cite{hellEJC35}, two main problems related to
matrix partitions are discussed.   The {\em
Characterization Problem} asks which patterns $M$
have the property that the number of minimal
obstructions to $M$-partition is finite.   The
{\em Complexity Problem} asks which patterns $M$
have the property that the $M$-partition problem
can be solved by a polynomial time algorithm.
Since both problems are very hard in the general
case, it has been a common practice to restrict them
to well behaved families of graphs.   Recall that a
graph is {\em perfect} if for every induced subgraph,
its chromatic number coincides with the size of a
maximum clique.   As perfect graphs have a unique
minimal obstruction for the $k$-colouring problem
(the complete graph on $k$ vertices), it is usual
to consider a family of perfect graphs, like cographs
or chordal graphs. It is known that the answer for the
characterization problem is always positive for cographs
\cite{damaschkeJGT14, federGTP2006} and split graphs
\cite{federDAM166}.   A graph $G$ is {\em chordal} if
every cycle in $G$ has a chord, or equivalently, if
every induced cycle in $G$ has length $3$. The class
of chordal graphs is one of the best understood graph
families where these two problems are still open.
Many papers deal with matrix partition problems for
chordal graphs, e.g., in \cite{hellDAM141}, a forbidden
subgraph characterization and a polynomial time
recognition algorithm is given for chordal graphs
admitting a partition into $k$ independent sets and
$\ell$ cliques, \cite{federTCS349} considers the list
version of the problem on chordal graphs, polarity of
chordal graphs is studied in \cite{ekimDAM156}, in
\cite{hellDM338} the $M$-partition problem is studied
on chordal graphs for a special family of patterns
called joining matrices.   In particular, this work
might be thought as a complement of \cite{federDM313},
where both the characterization problem and the
complexity problem is studied for small matrices
($m \times m$ with $m < 5$) on chordal graphs.   In
particular, they show that if $M$ is a matrix of
size $m < 4$, then $M$ has finitely many chordal
minimal obstructions, except for the following two
matrices, which have inifinitely many chordal
minimal obstructions.
\[
	M_1 = \left(
	\begin{array}{ccc}
	0    & \ast & \ast \\
	\ast & 0    & 1    \\
	\ast & 1    & 0
	\end{array}
	\right)
	\hspace{1cm}
	M_2 = \left(
	\begin{array}{ccc}
	0    & \ast & \ast \\
	\ast & 0    & 1    \\
	\ast & 1    & 1
	\end{array}
	\right)
\]
In \cite{federDM313}, an infinite family of chordal
minimal obstructions is exhibited for these patterns
(the same family works for both patterns).  Despite
the fact that the complete family of minimal obstructions
is obtained for other patterns, no efforts are made to
obtain the complete list for patterns $M_1$ and $M_2$.
The main objective of the present work is to provide
this missing list for $M_1$.   As we will see, among
the minimal obstructions there is a single infinite
family, all the other obstructions occur sporadically.

Let $\mathcal{F}$ be the family of chordal graphs containing
$F_i$ for $i \in \{1, \dots, 7\}$, and the members of the
family $\mathcal{F}_1$, depicted in Figure \ref{fig:ObsMin}.
Notice that the graphs in $\mathcal{F}_1$ consist of an odd
path of length at least $5$, together with an additional
vertex adjacent to every vertex of the path, except for the
first and last.

\begin{figure}[ht!]
	\centering
	\subfloat[$F_1$]{\scalebox{0.45}{\includegraphics{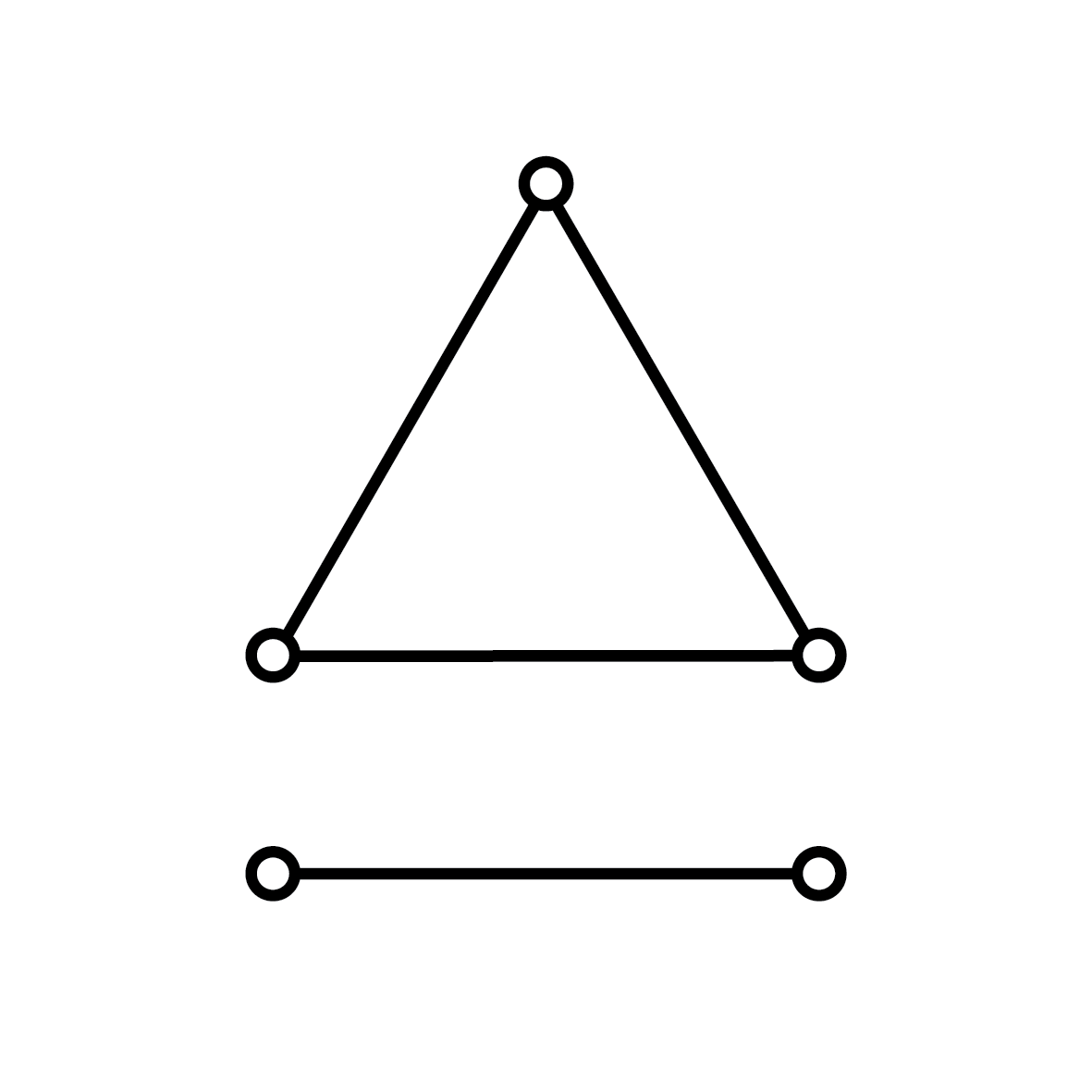}}}%
	\hfill
	\centering
	\subfloat[$F_2$]{\scalebox{0.45}{\includegraphics{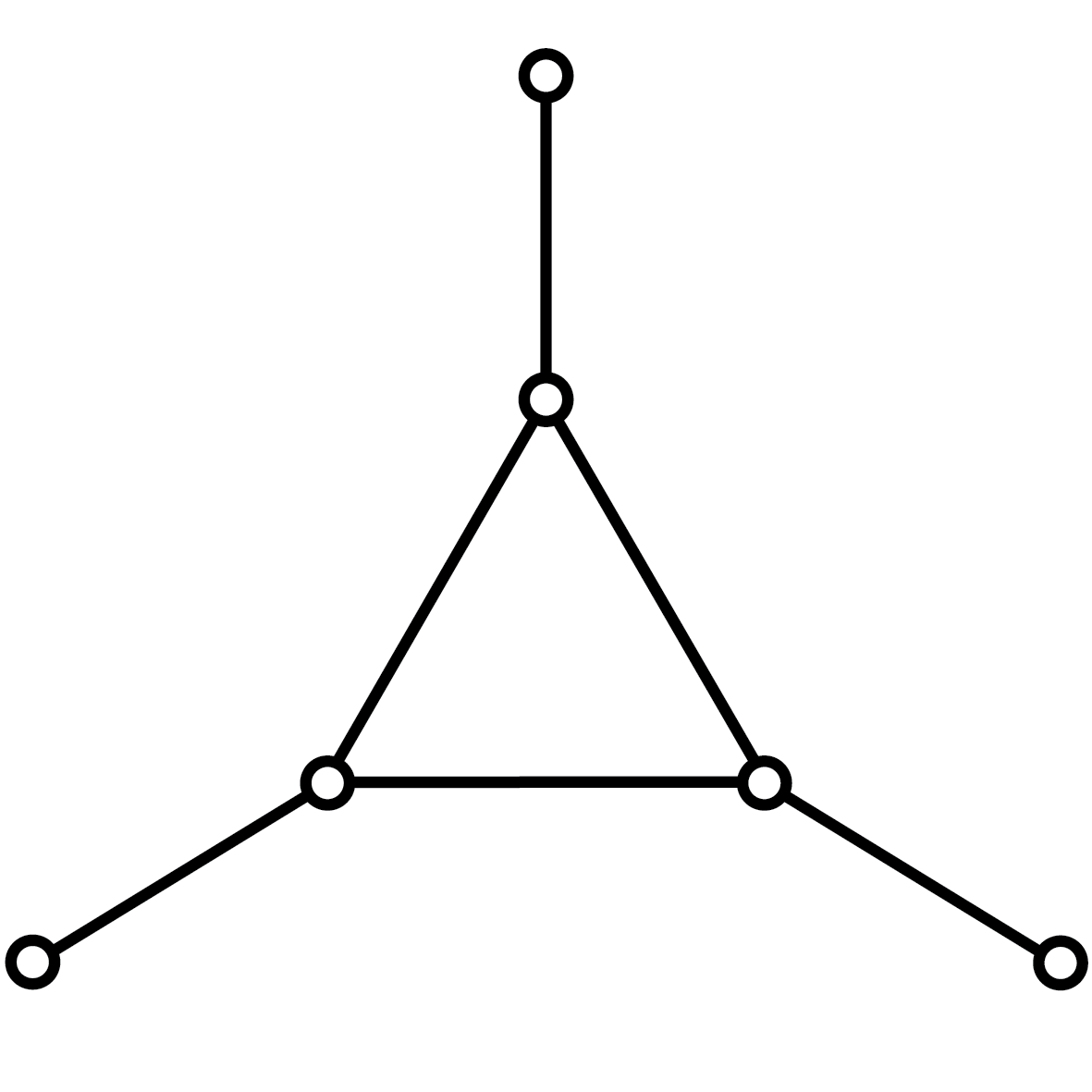}}}%
	\hfill
	\centering
	\subfloat[$F_3$]{\scalebox{0.45}{\includegraphics{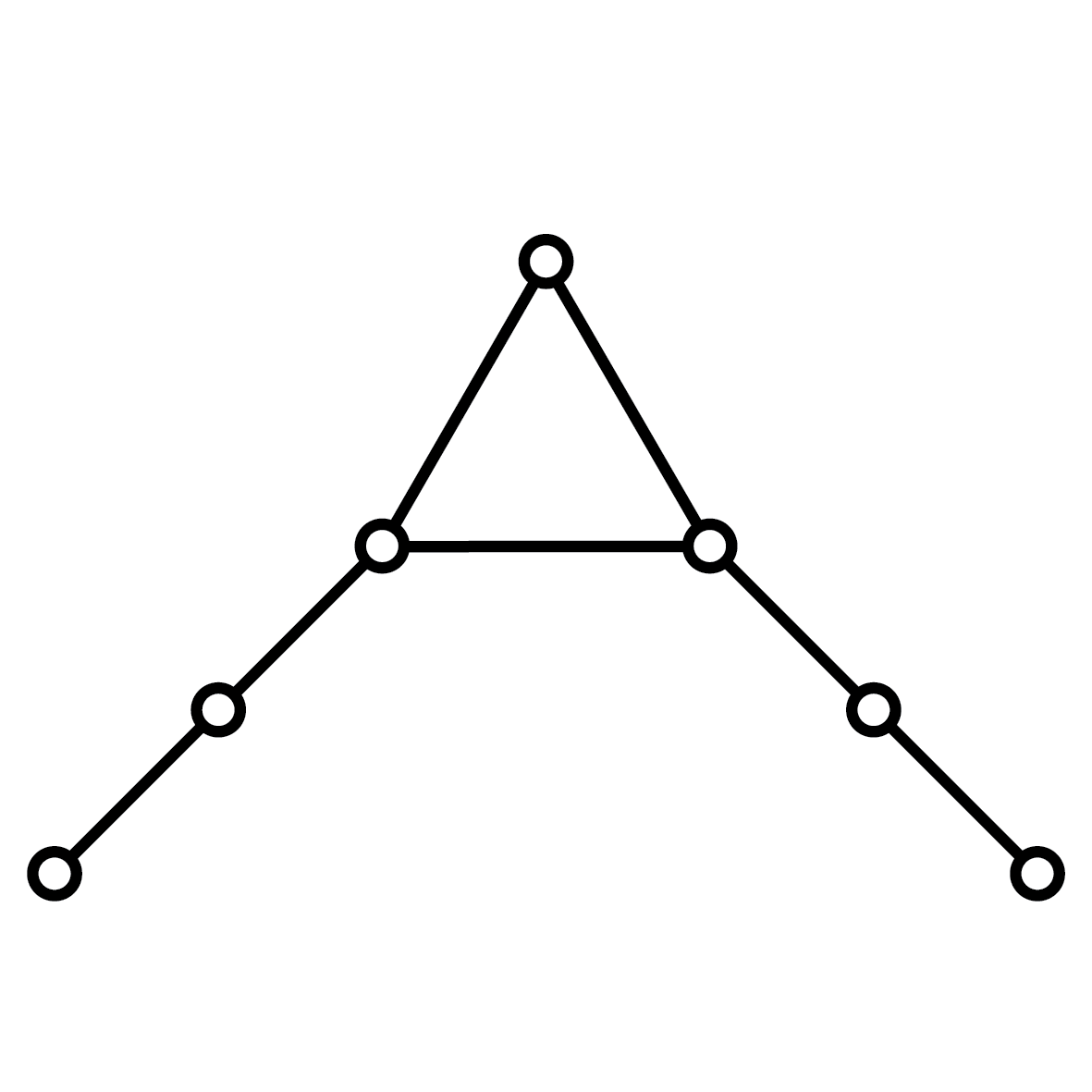}}}%
	\hfill
	\centering
	\subfloat[$F_4$]{\scalebox{0.45}{\includegraphics{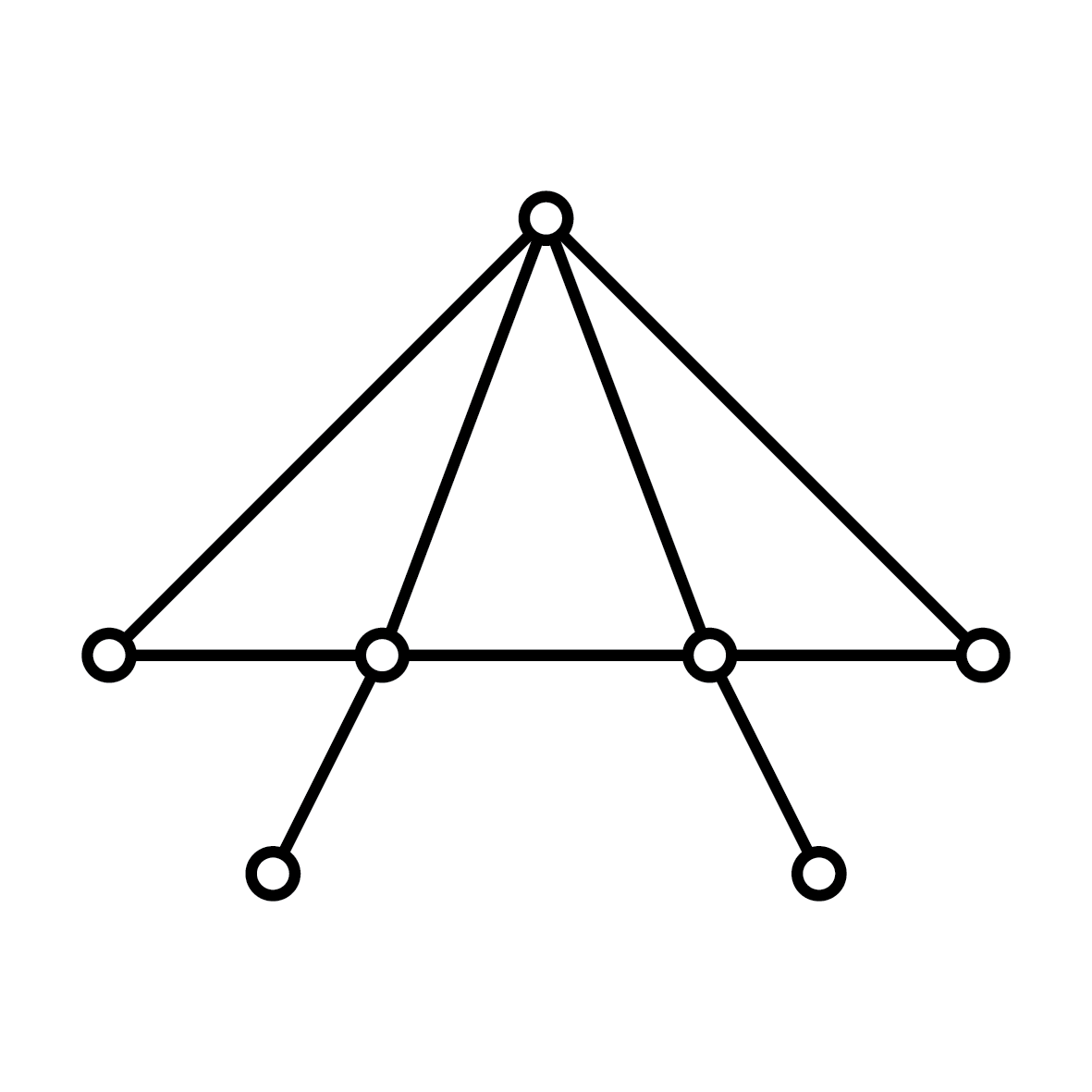}}}%
	\hfill
	\centering
	\subfloat[$F_5$]{\scalebox{0.45}{\includegraphics{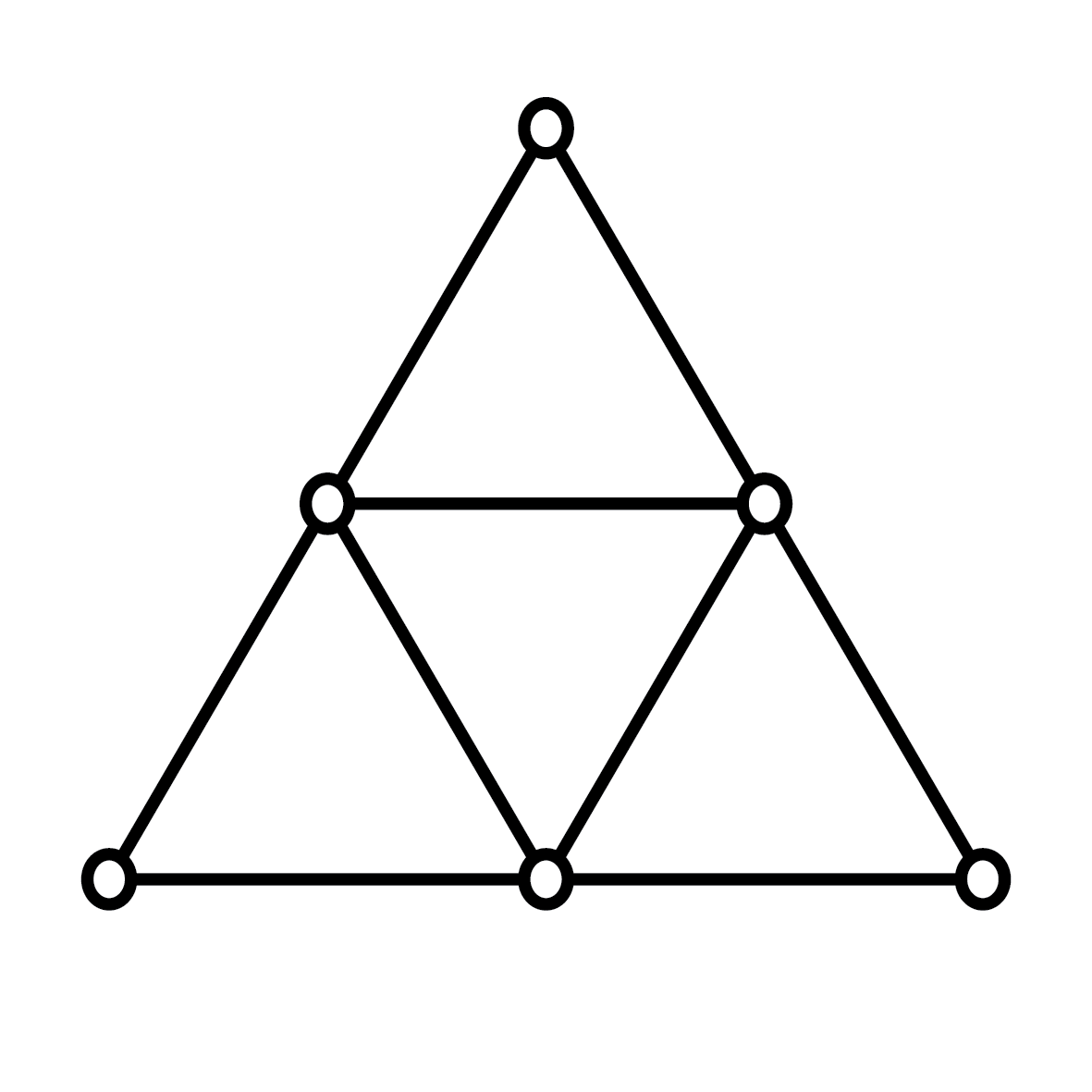}}}%
	\hfill
	\centering
	\subfloat[$F_6$]{\scalebox{0.45}{\includegraphics{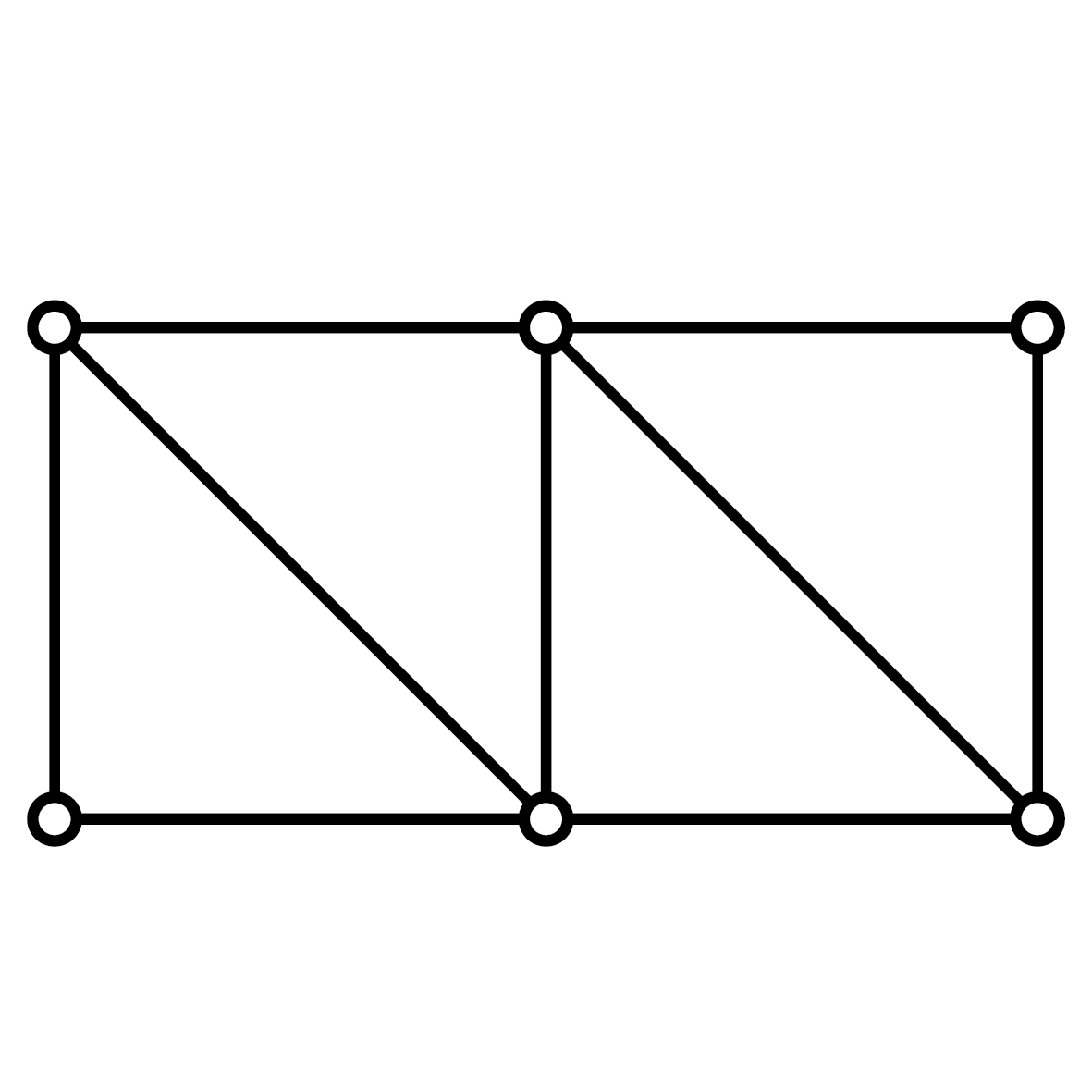}}}%
	\hfill
	\centering
	\subfloat[$F_7$]{\scalebox{0.45}{\includegraphics{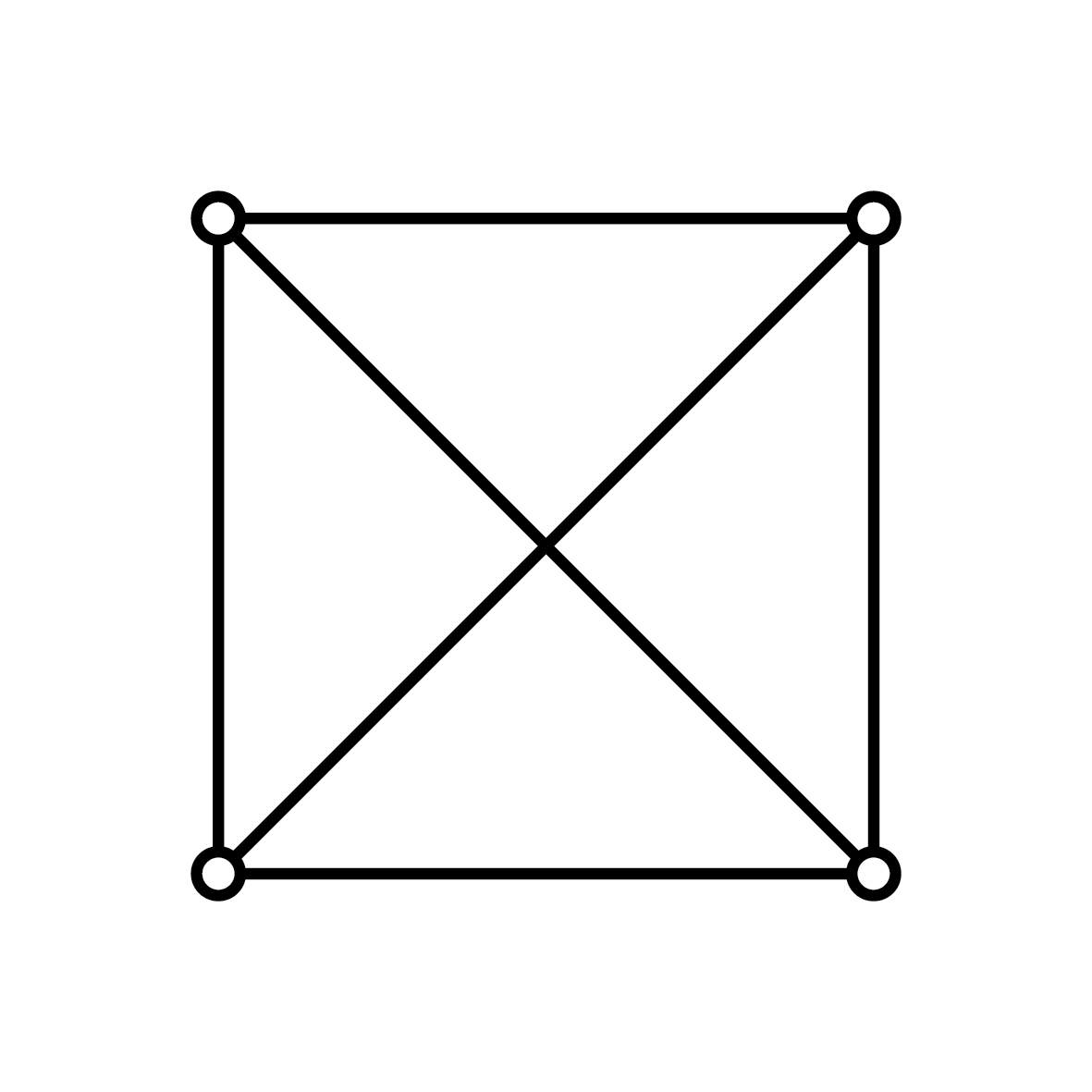}}}%
	\hfill
	\centering
	\subfloat[$\mathcal{F}_1$]{\scalebox{0.45}{\includegraphics{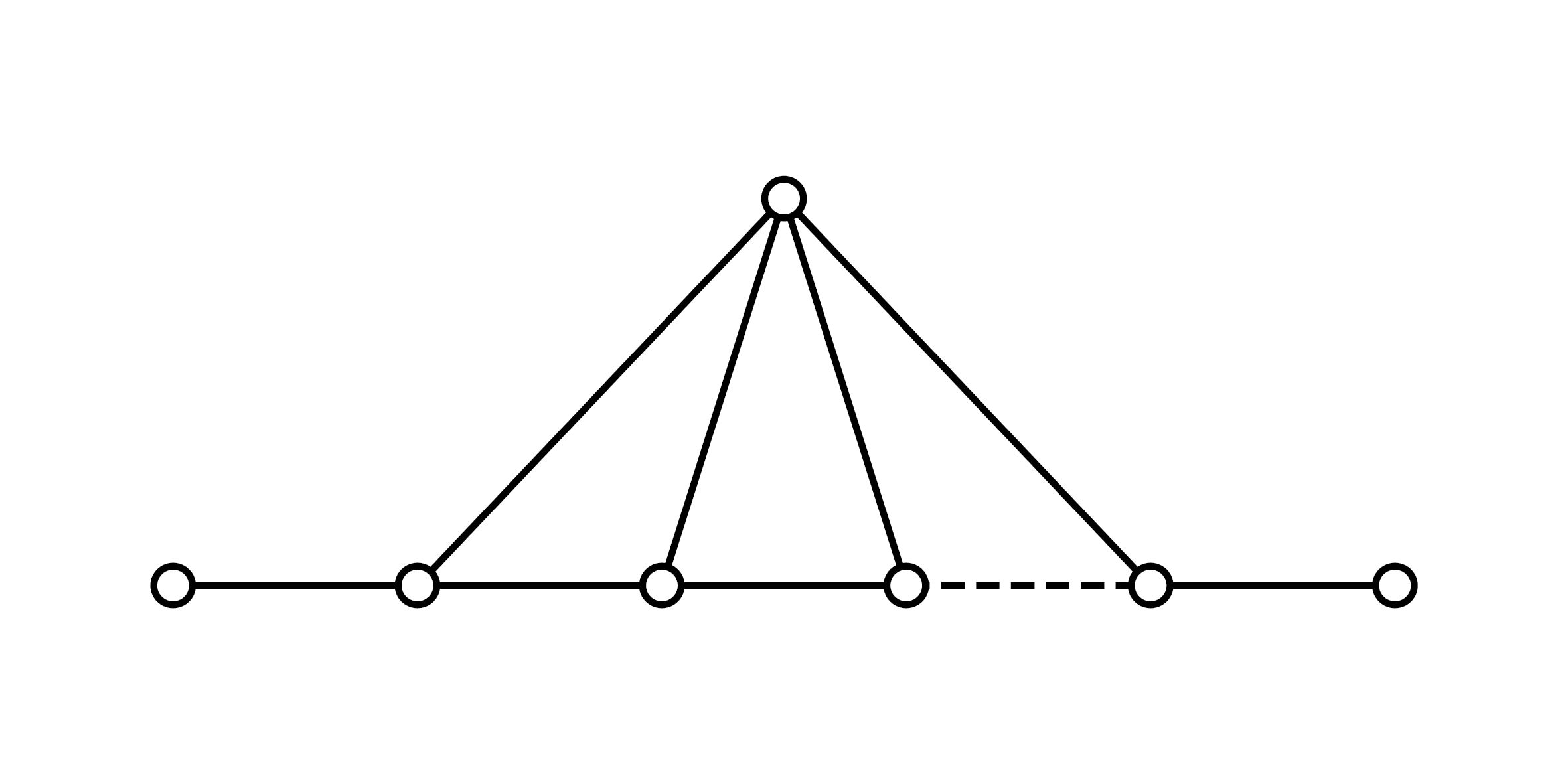}}}%
	\hfill
	\caption{The family $\mathcal{F}$.}
	\label{fig:ObsMin}
\end{figure}

The main result of this work is the following theorem.

\begin{theorem}\label{thm:M3}
If $G$ is a chordal graph, then $G$ admits an $M_1$-partition
if and only if it is $\mathcal{F}$-free.
\end{theorem}

For basic notions we refer the reader to \cite{bondy2008}.
Given a set $\mathcal{S}$ of graphs, we say that a graph
$G$ is {\em $\mathcal{S}$-free} if it does not contain any
member of $\mathcal{S}$ as an induced subgraph.   When
$\mathcal{S} = \{ S \}$ we will abuse notation and say that
$G$ is {\em $S$-free}.

The rest of this article is organized as follows.   In Section
\ref{sec:prelim}, technical results necessary to prove
Theorem \ref{thm:M3} are provided.   Section \ref{sec:main}
is devoted to prove Theorem \ref{thm:M3}.   We conclude
presenting conclusions and future lines of work in Section
\ref{sec:conc}.

\section{Preliminary results}
\label{sec:prelim}

In this section we will obtain some basic technical
results, necessary for the proof of our main theorem.

\begin{lemma}
\label{lem:F5}
Let $G$ be a chordal graph. If $G$ contains $F_5$ as a
subgraph, then it contains $F_5$ or $F_7$ as an induced
subgraph.
\end{lemma}

\begin{proof}
Consider a copy of $F_5$ as a subraph of $G$ with the
configuration shown in Figure \ref{fig:Lema1Fig1}.

If there are no further edges, this is, if the dotted
edges are missing in the figure, then this is is an
induced copy of $F_5$.   If any of the edges  $14, 25$
or $36$ is present, then we have $F_7$ as an induced
subgraph.  If $13$ is present, then we also have the
cycle $(1,3,4,6,1)$ and, since $G$ is chordal, either
$14$ or $36$ must be an edge of $G$, and we are in the
previous case.   The cases when $15$ or $35$ are present
are analogous.

\begin{figure}
	\centering
	\includegraphics[height=4cm]{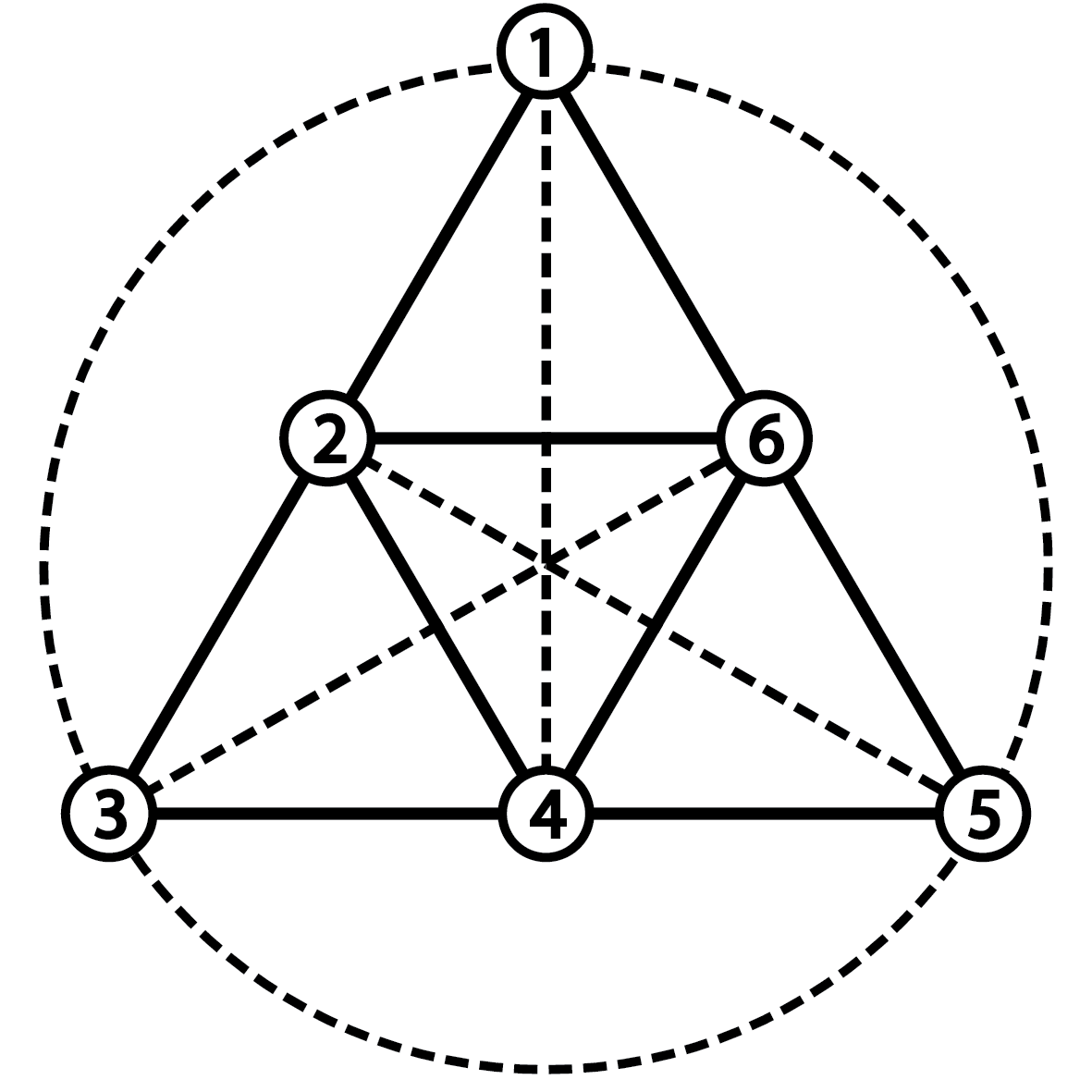}
	\caption{Configuration of $F_5$ for Lemma \ref{lem:F5}.}
	\label{fig:Lema1Fig1}
\end{figure}
\end{proof}

Our following lemma has a similar flavour.

\begin{lemma}
\label{lem:F6}
Let $G$ be a chordal graph. If $G$ contains $F_6$ as a
subgraph, then it contains $F_6$ or $F_7$ as an induced
subgraph.
\end{lemma}

\begin{proof}
Consider a copy of $F_6$ as a subraph of $G$ with the
configuration shown in Figure \ref{fig:Lema2Fig2}.

If there are no further edges, this is, if the dotted
edges are missing in the figure, then this is is an
induced copy of $F_6$.   If any of the edges $16,
24$, or $35$ is present, then we have $F_7$ as an
induced subgraph. If the edge $13$ is present, the cycle
$(1,5,6,3,1)$, is contained in $G$, and it follows from
the chordality of $G$ that either $16$ or $35$ is an edge,
which takes us to the previous case. An analogous case
happens if $46$ is an edge of $G$. Finally, if $34$ is
present, then $(1,2,3,4,1)$ is a cycle in $G$, and thus
$13$ or $24$ must be an edge of $G$.   Again, this falls
in one of the previous cases.

\begin{figure}[ht!]
	\centering
	\includegraphics[height=4cm]{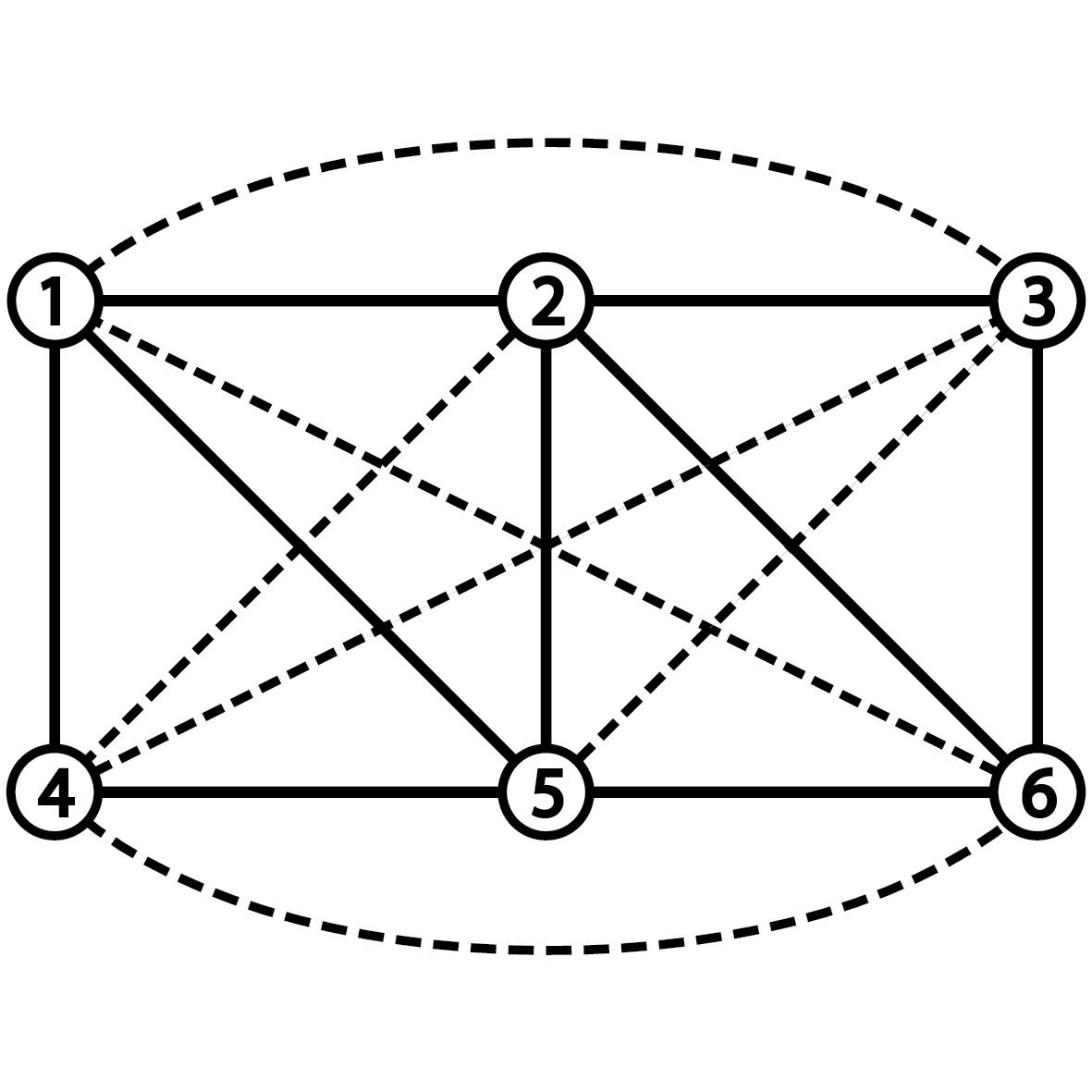}
	\caption{Cases for Lemma \ref{lem:F6}.}
	\label{fig:Lema2Fig2}
\end{figure}
\end{proof}

Although again, similar in nature, our following lemma
deals with a graph not in $\mathcal{F}$.   For consistency,
we will denote $2K_3$ by $F_0$.

\begin{lemma}
\label{lem:F0}
Let $G$ be a chordal graph. If $G$ contains $F_0$ as a
subgraph, then it contains $F_1$, $F_6$ or $F_7$ as an
induced subgraph.
\end{lemma}

\begin{proof}
Let $A$ and $B$ be the two triangles in $F_0$. If there
are less than three edges between $A$ and $B$, then, since
$G$ is chordal, they cannot form a matching (otherwise
there would be a chordless $C_4$ in $G$).   Thus, $G$
contains an induced copy of $F_1$.

So, suppose that there are at least three edges between
$A$ and $B$.   If three edges are incident in the same
vertex $v_0$ of $A$, then the vertices of $B$ together
with $v_0$ induce an $F_7$ (left side of Figure
\ref{fig:LemaF0}). Else, at least two of the edges from
$A$ to $B$ form a matching, and hence a $C_4$ with one
edge from each triangle.   Again, the chordality of $G$
implies that at least one of the diagonals of the $C_4$
exists, and thus, the result follows from Lemma
\ref{lem:F6} (right side in Figure \ref{fig:LemaF0}).

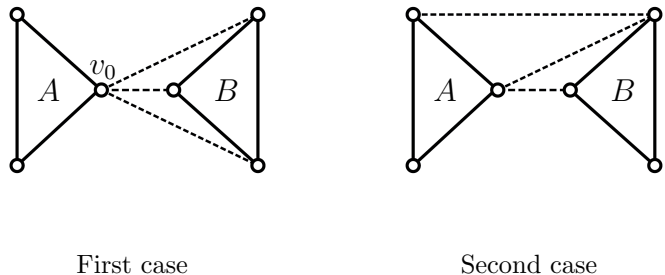
\begin{figure}
\centering
	\subfloat[First case]{
  \includegraphics[height=4cm]{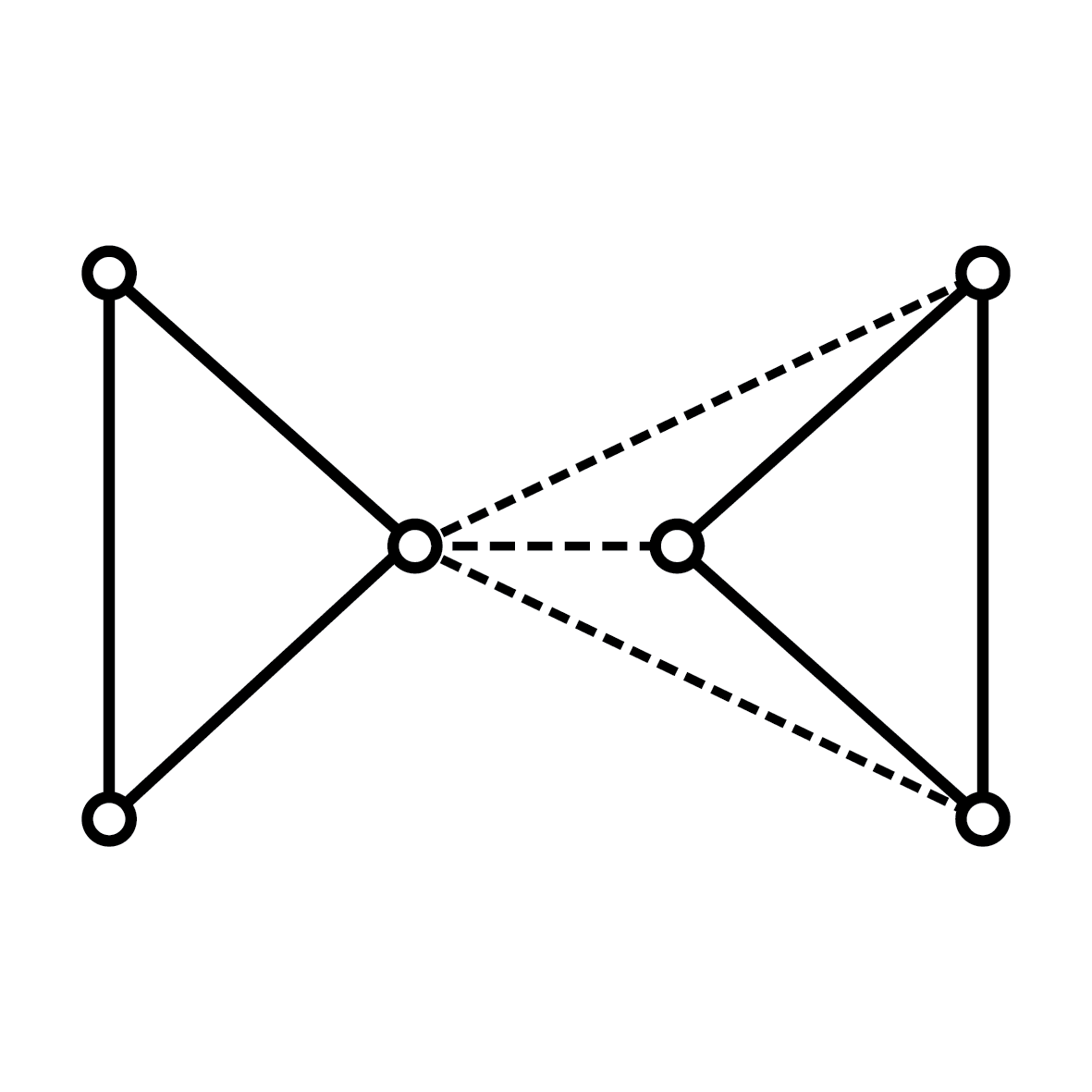}
		\put(-95, 53){$A$}
		\put(-28, 53){$B$}
		\put(-75,63){$v_0$}
	}%
	\hspace{1cm}
	\centering
	\subfloat[Second case]{
	\includegraphics[height=4cm]{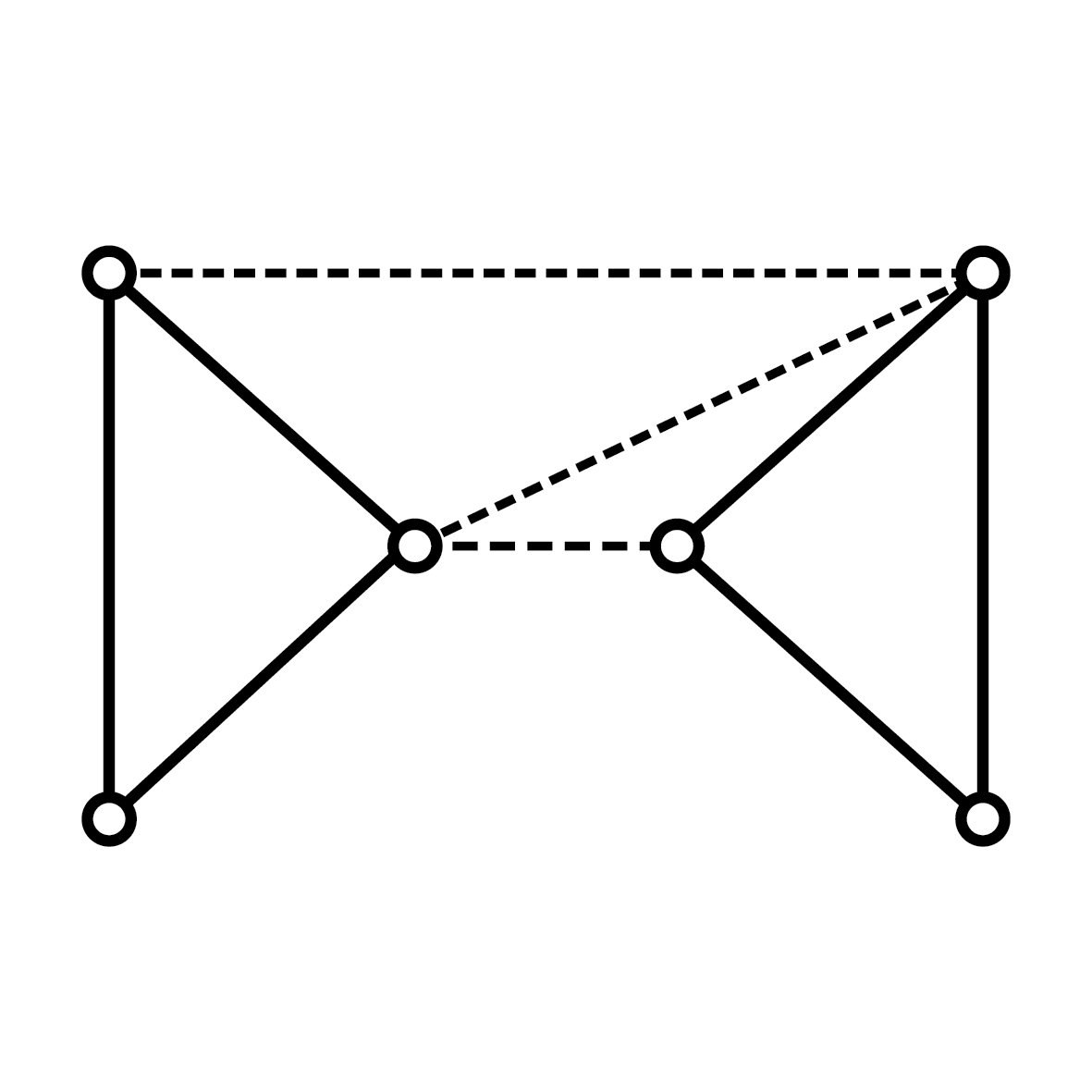}
		\put(-95, 53){$A$}
		\put(-28, 53){$B$}
	}%
	\caption{Cases for Lemma \ref{lem:F0}.}
	\label{fig:LemaF0}
\end{figure}
\end{proof}

The following result characterizes chordal graphs
which are bipartizable by removing a single vertex,
in terms of forbidden subgraphs.

\begin{lemma}
\label{lem:Bipart}
If $G$ is a chordal graph, then there exists a vertex
$v_0 \in V$ such that $G-v_0$ is bipartite if and only
if $G$ contains neither $F_0$, $F_5$, nor $F_7$ as a
subgraph (not necessarily induced).
\end{lemma}

\begin{proof} The result is clear if $G$ contains a
copy of $F_0$, $F_5$ or $F_7$.   For the converse we
will proceed by contradiction.

Aiming for a contradiction, suppose that $G$ contains
neither $F_0$, $F_5$ nor $F_7$ as subgraphs, and for
every vertex $v$ of $G$, we have that $G-v$ is not
bipartite. It follows that $G$ contains at least two
different triangles $A$ and $B$.   Since $G$ does not
contain $F_0$ as an induced subgraph, then $A$ and $B$
must share one, or two vertices.

Suppose first that $A$ and $B$ share a single vertex
$v_0$.   Since $G-v_0$ is not bipartite, then there
is another triangle $C$ in $G$ which does not contain
the vertex $v_0$, but shares vertices with both $A$
and $B$.   If $C$ shares exactly on vertex with $A$,
and exactly on vertex with $B$, then its third vertex
should be a new vertex, neither in $A$ nor in $B$,
which would give us a copy of $F_5$ contained in $G$,
which cannot happen (see leftmost graph in Figure
\ref{fig:LemaBipart}).   If $C$ shares two vertices
with $A$ and a single vertex with $B$, then $G$ would
contain a copy of $F_7$ as a subgraph (center graph in
Figure \ref{fig:LemaBipart}).   Thus, this case is also
impossible.

If $A$ and $B$ share two vertices, say $v_1$ and $v_2$,
then, since $G-v_1$ is not bipartite, there exists a
triangle $C$ in $G$ not using the vertex $v_1$ and such
that it shares two vertices with $A$ and two vertices
with $B$ (otherwise we would be in the previous case).
But for this to happen, there should be an edge joining
the vertices in $A$ and $B$ different from $v_1$ and
$v_2$ (rightmost graph in Figure \ref{fig:LemaBipart}),
which creates a copy of $F_7$, a contradiction.

Since a contradiction is reached in each case, we conclude
that there is a vertex $v_0$ in $G$ such that $G-v_0$ is
bipartite.

\begin{figure}
	\centering
	\subfloat{
		\includegraphics[height=4cm]{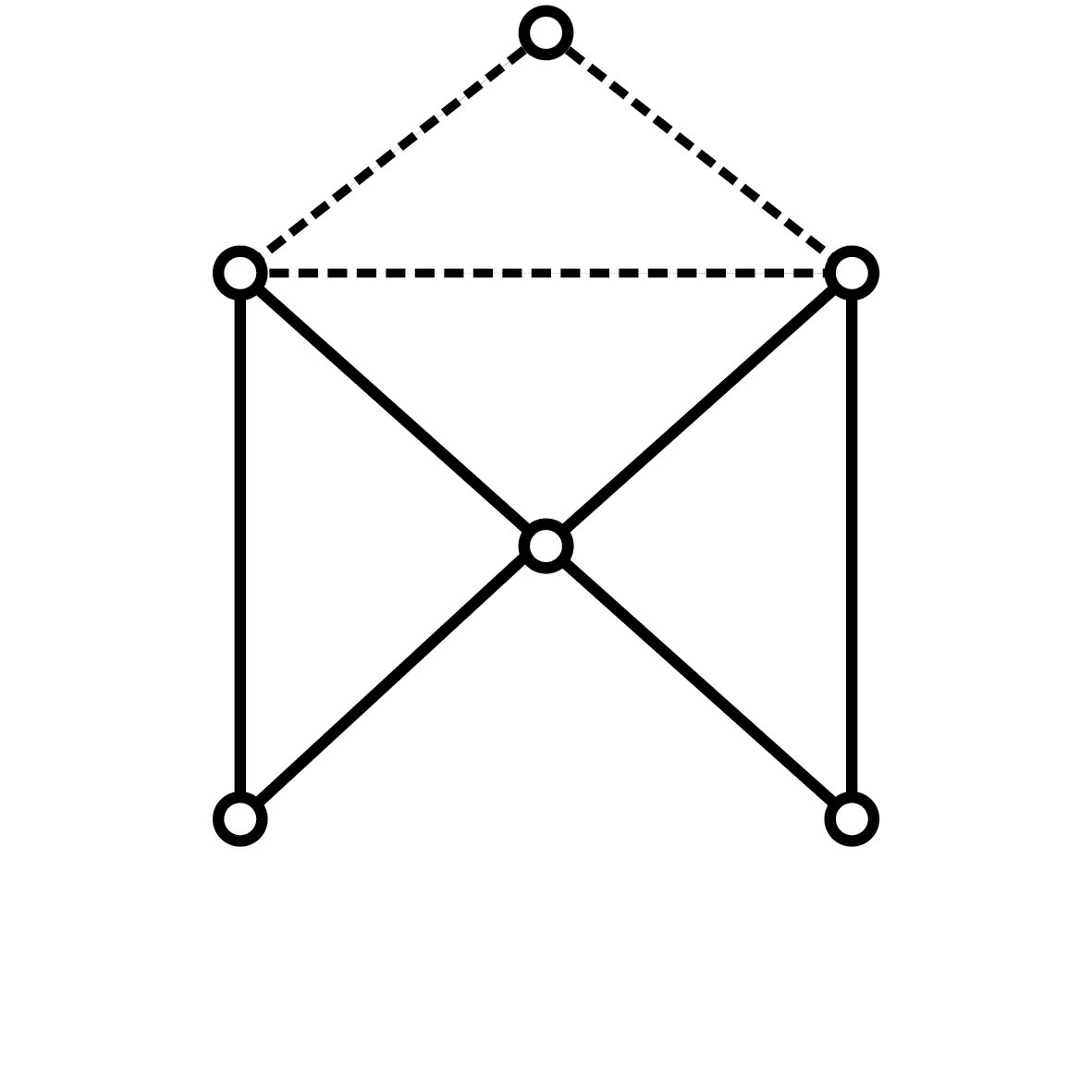}
			\put(-83, 53){$A$}
			\put(-42, 53){$B$}
			\put(-61, 92){$C$}
			\put(-61, 45){$v_0$}
	}%
	\hfill
	\centering
	\subfloat{
		\includegraphics[height=4cm]{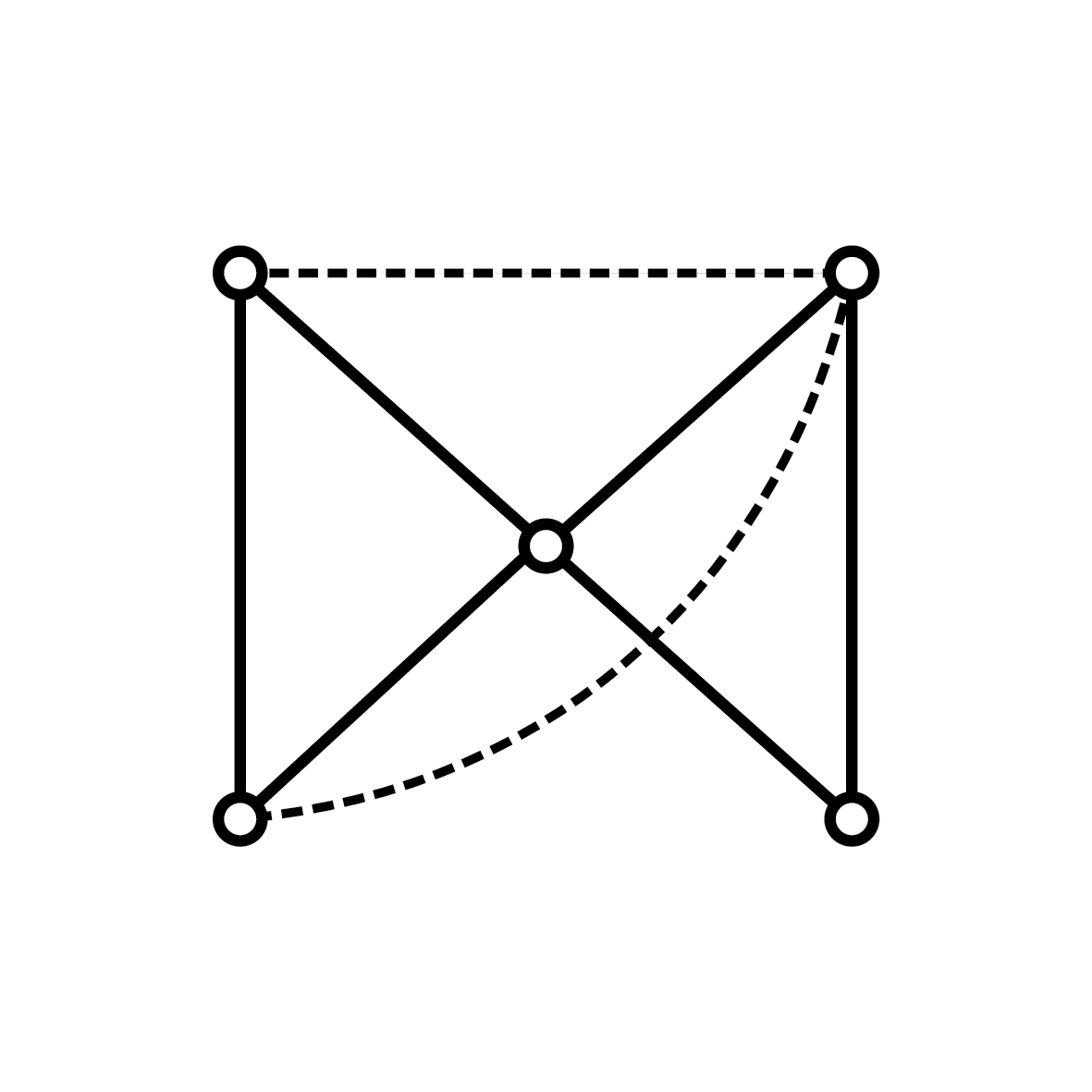}
			\put(-83, 53){$A$}
			\put(-42, 53){$B$}
			\put(-61, 70){$C$}
			\put(-61, 45){$v_0$}
	}%
	\hfill
	\centering
	\subfloat{
		\includegraphics[height=4cm]{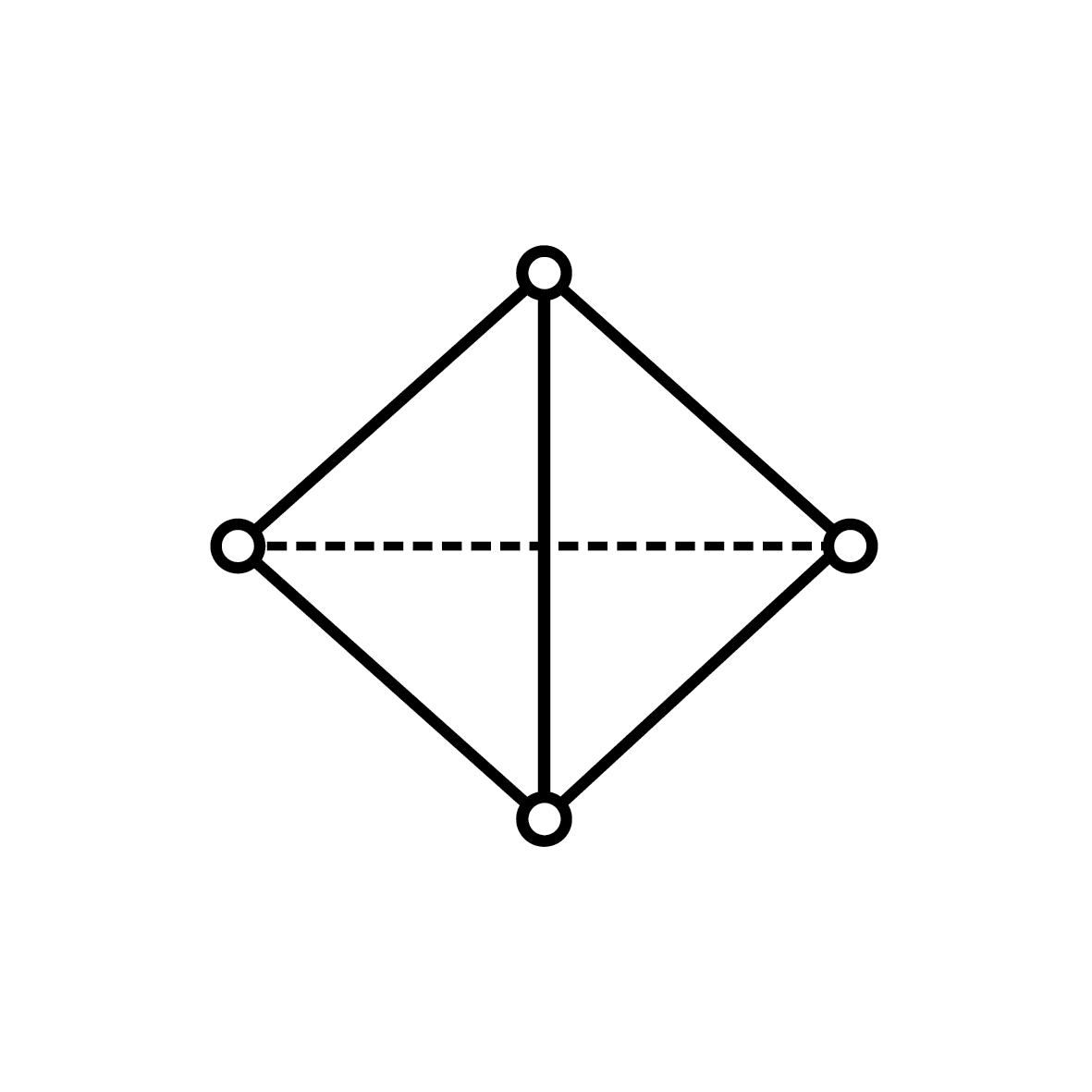}
			\put(-76, 60){$A$}
			\put(-50, 60){$B$}
			\put(-61, 92){$v_1$}
			\put(-61, 18){$v_2$}
	}%
	\vspace{-0.75cm}
	\caption{Cases for Lemma \ref{lem:Bipart}.}
	\label{fig:LemaBipart}
\end{figure}
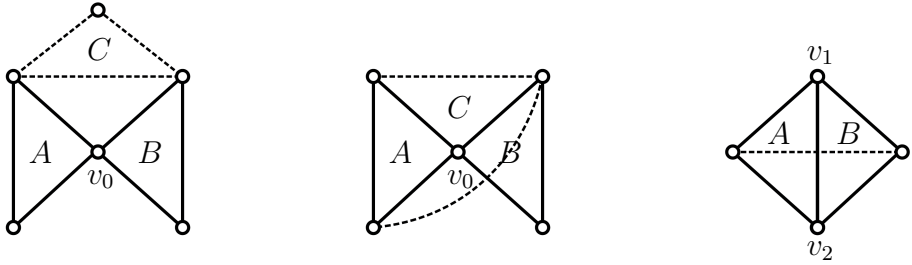
\end{proof}

Notice that Lemma \ref{lem:Bipart} can be easily
adapted to be stated in terms of forbidden induced
subgraphs, as our next result shows.   Besides the
graphs in family $\mathcal{F}$, we need two additional
graphs, $F_{01}$ and $F_{02}$, depicted in Figure
\ref{fig:F01F02}.

\begin{figure}
	\centering
	\subfloat{
		\includegraphics[height=4cm]{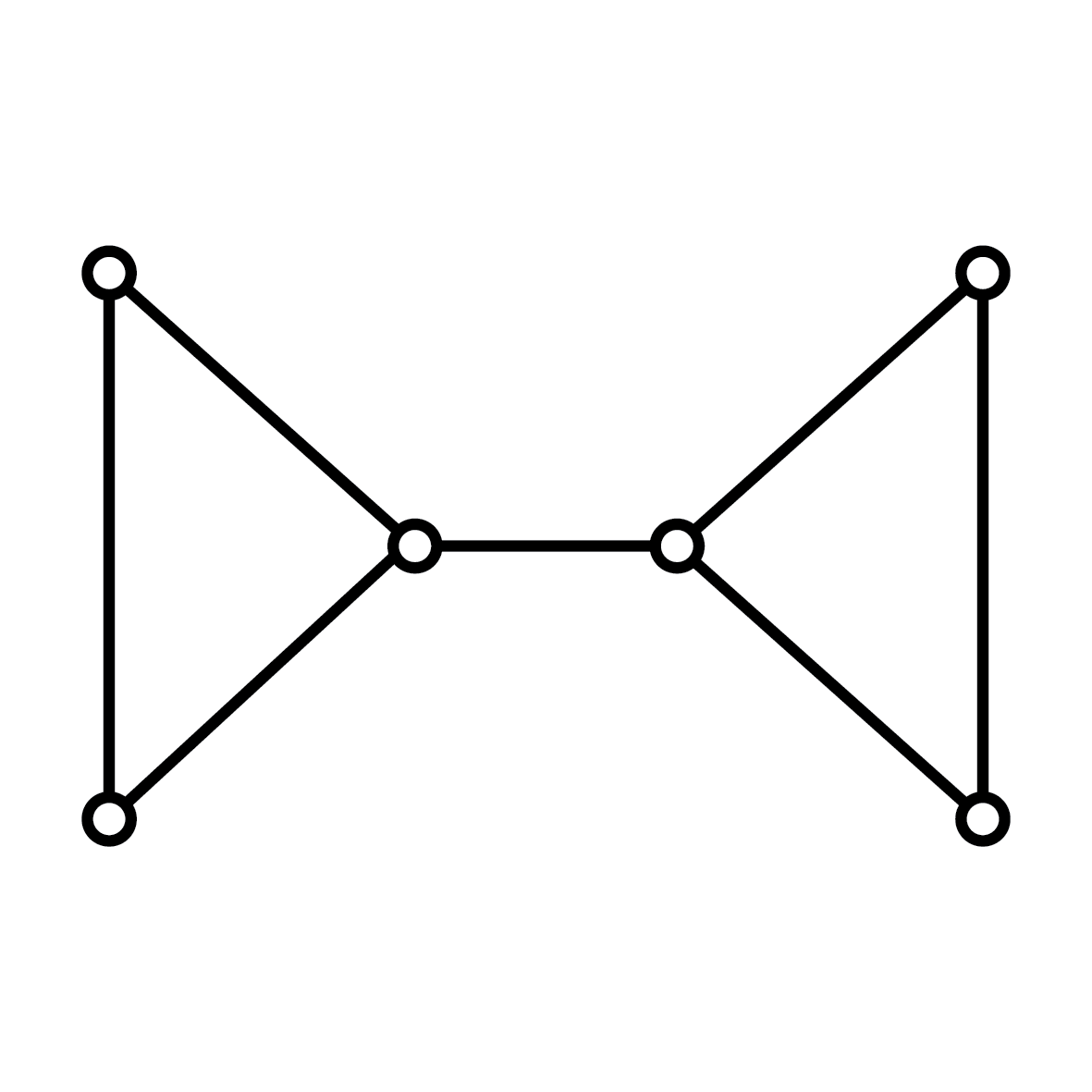}
	}%
	\hspace{1cm}
	\centering
	\subfloat{
		\includegraphics[height=4cm]{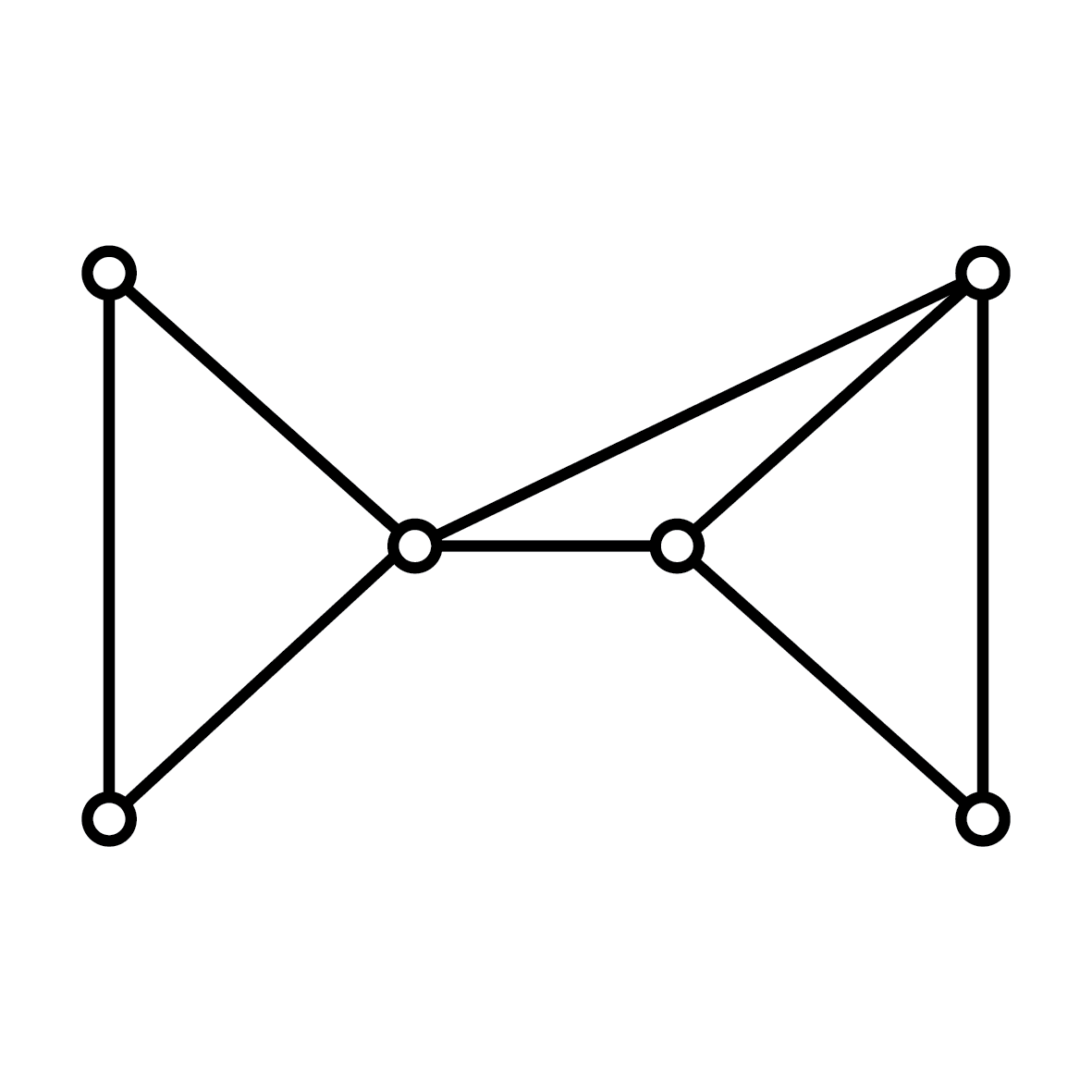}
	}%
	\vspace{-0.75cm}
	\caption{Graphs $F_{01}$ (left) and $F_{02}$
	(right) for Lemma \ref{lem:BipartInduc}.}
	\label{fig:F01F02}
\end{figure}

\begin{lemma}
\label{lem:BipartInduc}
If $G$ is a chordal graph, then there exists a vertex
$v_0 \in V$ such that $G-v_0$ is bipartite if and only
if $G$ is $\{ F_5, F_6, F_7, F_0, F_{01}, F_{02}
\}$-free.
\end{lemma}

\begin{proof}
We will only prove the non-trivial implication.   Suppose
that $G$ is a chordal $\{ F_5, F_6, F_7, F_0, F_{01},
F_{02} \}$-free graph.   As $G$ is $\{ F_5, F_7 \}$-free,
it follows from Lemma \ref{lem:F5}, that $G$ does not
contain $F_5$ as a subgraph.   Since $F_7$ is complete,
having it as an induced subgraph is equivalent to having
it as a subgraph.

In the proof of Lemma \ref{lem:F0}, in the case where
there are at least three edges between $A$ and $B$, we
concluded that $G$ contains an induced copy of either
$F_6$ or $F_7$.   It is not hard to verify that in the
case where there are at most two edges, $G$ must contain
an induced copy of either $F_0, F_{01}$ or $F_{02}$.

Thus, if $G$ is $\{ F_5, F_6, F_7, F_0, F_{01},
F_{02} \}$-free, then it contains neither $F_0, F_5$ nor
$F_7$ as a subgraph. The desired result follows from
Lemma \ref{lem:Bipart}.
\end{proof}

Clearly, if $F_1$ is not an induced subgraph of $G$,
then neither $F_0,F_{01}$ and $F_{02}$ are. Thus, a
direct application of Lemma \ref{lem:BipartInduc}
produces the next result.

\begin{lemma}
\label{lem:equiv}
If $G$ is a chordal graph, then $G$ is $\mathcal{F}$-free
if and only if $G$ is $(\{ F_1, F_2, F_3, F_4 \} \cup
\mathcal{F}_1$)-free and there exists a vertex $v_0$ in
$G$ such that $G-v_0$ is bipartite.
\end{lemma}


\section{Main result}
\label{sec:main}

For a chordal graph $G$ we define the set
$\mathcal{B}_G$ as
\[
	\mathcal{B}_G = \{ v \in V \colon\ G-v
	\textnormal{ is bipartite} \},
\]
then every element of $\mathcal{B}$ must belong to
every triangle in $G$, and hence, if $G$ is not
bipartite, then $|\mathcal{B}_G| \le 3$.

Notice that if a graph $G$ admits an $M_1$-partition,
then any graph obtained from $G$ by adding isolated
vertices also admits an $M_1$-partition; it suffices
to place all the isolated vertices in $V_1$, the
independent set without further restrictions.   Also,
if $G$ is a chordal, non-bipartite, $\mathcal{F}$-free
graph, then at least one component of $G$ contains a
triangle, and being $F_1$-free, every other component
of $G$ must be an isolated vertex.   Thus, in order to
prove our main result, we may only consider connected
graphs.

We are now ready to prove our main result.  We will
use Lemma \ref{lem:equiv} to prove a slightly different
equivalent form of Theorem \ref{thm:M3}.

\begin{theorem}
\label{thm:MainTheorem}
If $G$ is a chordal graph, then $G$ admits an
$M_1$-partition if and only if $G$ is $(\{ F_1, F_2,
F_3, F_4 \} \cup \mathcal{F}_1)$-free and there exists
a vertex $v_0 \in V$ such that $G-v_0$ is bipartite.
\end{theorem}

\begin{proof}
Let $G$ be a graph with an $M_1$-partition $(V_1,
V_2, V_3)$ (recall that $V_2$ and $V_3$ are completely
adjacent).   If $G$ is bipartite, then the conditions
of the theorem clearly hold.   Else, $V_i$ is non-empty
for $i \in \{ 1, 2, 3 \}$.   If $|V_2|, |V_3| \ge 2$,
then $G$ would contain a chordless $4$-cycle, contradicting
the chordality of $G$.   Thus, we will assume without loss
of generality that $V_3 = \{ v_3 \}$.   Hence, $G-v_3$ is
a bipartite graph.

For the remaining condition, it suffices to verify that
none of $F_1, F_2, F_3, F_4$, nor any graph in
$\mathcal{F}_1$ admits an $M_1$-partition.   These
verifications are simple yet tedious, so we will omit
them.   This concludes the proof for the necessity.

For the suficiency, if $G$ is bipartite, then it clearly
admits an $M_1$-partition.   Else, as we discussed at
the beginning of this section, we can assume that $G$
is connected.   Since $1 \le |\mathcal{B}_G| \le 3$ we
will consider three cases, one for each possible
cardinality of $\mathcal{B}_G$.   We are looking for
an $M_1$-partition $(V_1, V_2, V_3)$ as described in
the Introduction; we will describe such a partition by
by colouring the vertices of $G$ with colours $\{ 1,
2, 3 \}$, which correspond to the parts of the partition.

{\bf Case 1:} If $\mathcal{B}_G = \{ v_0,v_1,v_2 \}$,
then, as each of $v_0, v_1$ and $v_2$ must be in every
triangle of $G$, then there is only one triangle in $G$,
namely the one induced by $\mathcal{B}_G$.

Since $G$ is chordal, $G-v_i$ is acyclic for every $i
\in \{ 0, 1, 2 \}$, as otherwise there would be a
triangle not using $v_i$.   Let $G_i$ be the connected
component of $G - \{ v_j, v_k \}$ where $\{ i, j, k \}
= \{ 0, 1, 2 \}$.   Then, $G_i$ is a tree for $i \in
\{ 0, 1, 2 \}$, and since $G$ is $F_2$-free, then $V_{G_i}
= \{ v_i \}$ for at least one $i \in \{ 0, 1, 2 \}$.
Assume without loss of generality that $V_{G_0} =
\{ v_0 \}$.

If we consider $G_i$ as a tree rooted at $v_i$, then
the fact that $G$ is $F_3$-free implies that $G_1$
and $G_2$ cannot both have height greater than $1$.
Suppose without loss of generality that $G_2$ has
height less than $2$.   If $G_1$ has height greater
than $1$, then it must be $2$, otherwise $G$ would
contain $F_1$ as an induced subgraph.   See the
leftmost graph in Figure \ref{fig:MainTheorem}
for a depiction of the current configuration.

\begin{figure}
	\centering
	\subfloat{
		\includegraphics[height=4cm]{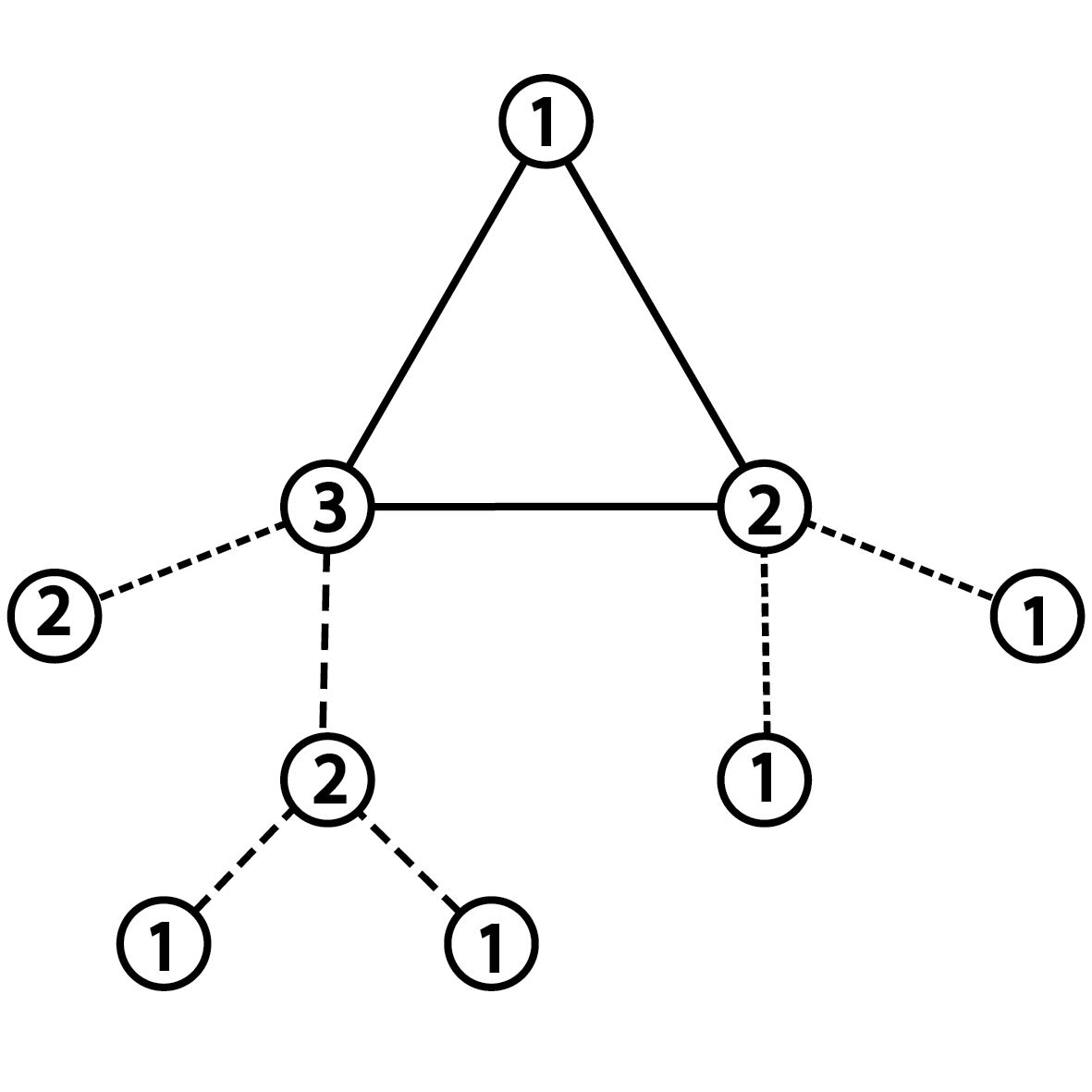}
		\put(-63,72){$A$}
		\put(-63,110){$v_0$}
		\put(-92,69){$v_1$}
		\put(-32,69){$v_2$}
	}
	\centering
	\hfill
	\subfloat{
		\includegraphics[height=4.1cm]{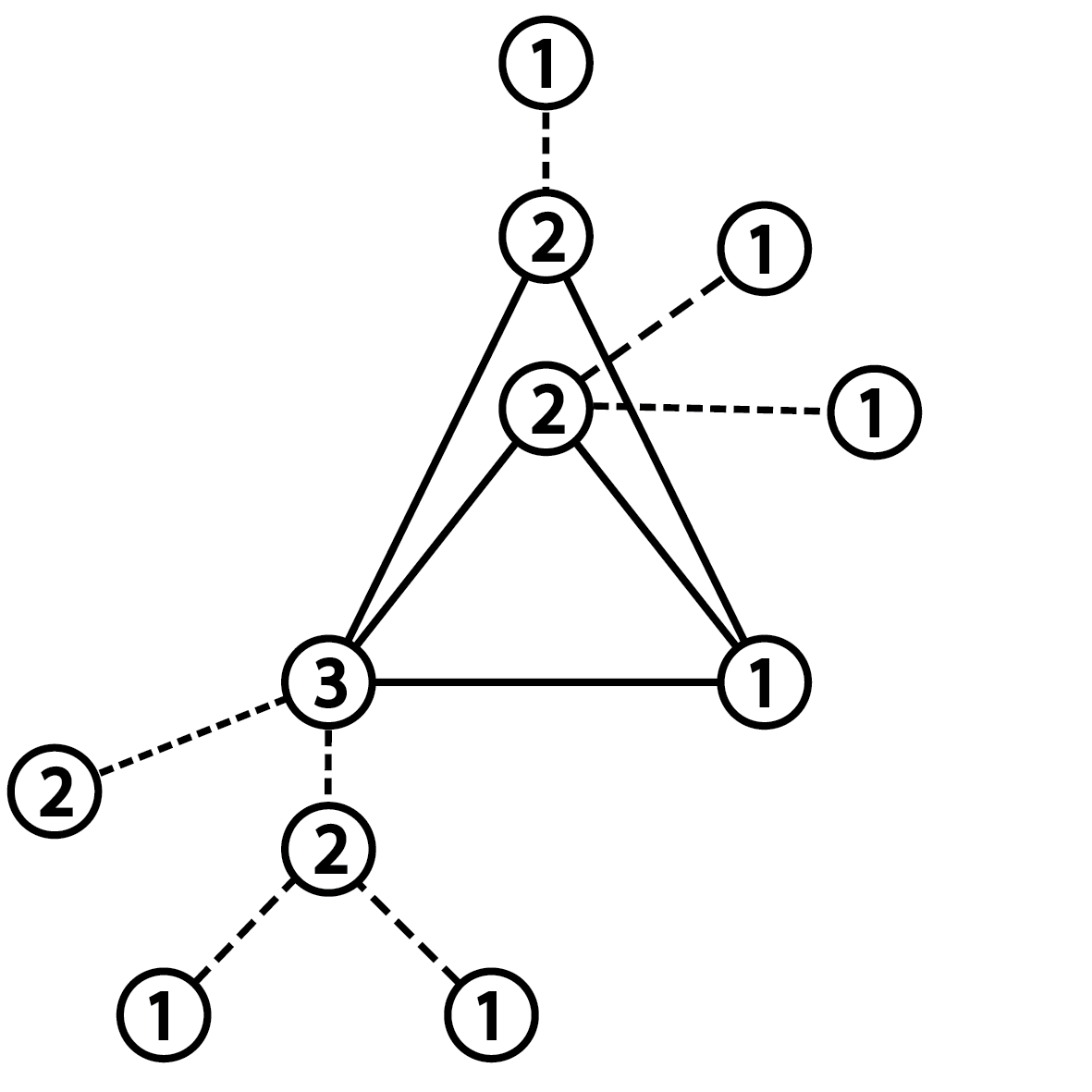}
		\put(-64,46){$A$}
		\put(-94,51){$v_1$}
		\put(-33,51){$v_2$}
		\put(-63,60){$v_0$}
	}
	\centering
	\hfill
	\subfloat{
		\includegraphics[height=4cm]{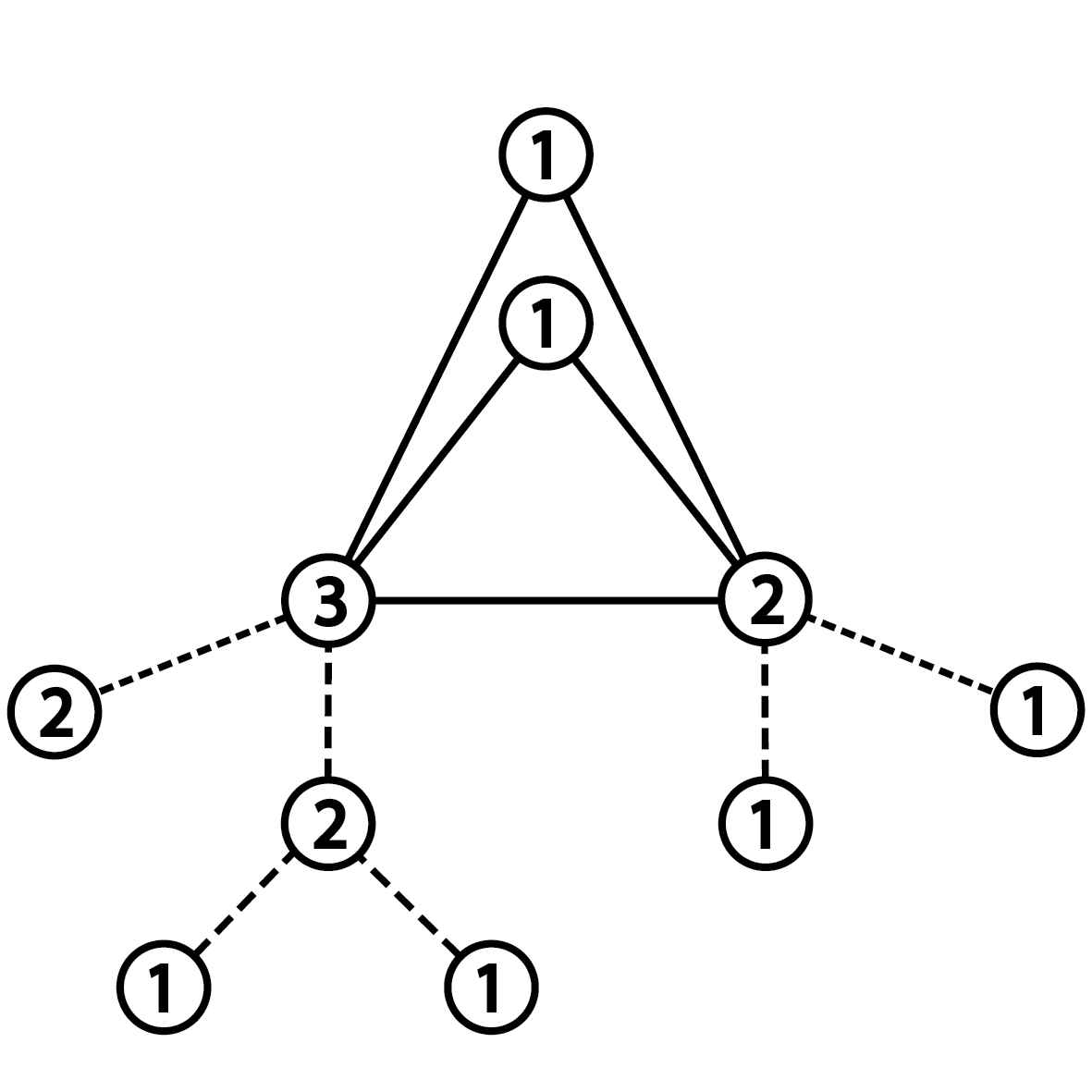}
		\put(-92,59){$v_1$}
		\put(-32,59){$v_2$}
	}
	\caption{Cases for Theorem \ref{thm:MainTheorem}.}
	\label{fig:MainTheorem}
\end{figure}
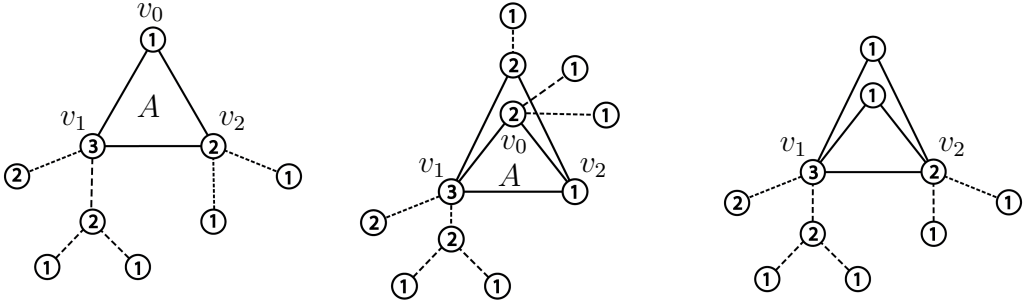

Thus, we colour $v_0$ with colour $1$, $v_1$ with colour
$3$ and $v_2$ with colour $2$.   The neighbours of $v_2$
in $G_2$ are coloured with $1$, the neighbours of $v_1$
in $G_1$ are coloured $2$, and the vertices at depth $2$
in $G_1$ are coloured $1$.   It follows from the structure
of $G$ in this case that this colouring is an
$M_1$-partition.

{\bf Case 2:} Suppose that $\mathcal{B}_G = \{ v_1, v_2 \}$.
If there were exactly one triangle in $G$, then all of its
vertices would be in $\mathcal{B}_G$, and we would be back
to Case 1.   So, there are at least two different triangles
in $G$.   Also, recall that $v_1$ and $v_2$ are in every
triangle of $G$.   Let $v_0$ be a vertex such that $A =
(v_0, v_1, v_2, v_0)$ is a triangle in $G$.   Let $G_1$
be the connected component of the subgraph of $G$ obtained
by deleting $v_2$ and all the vertices that form a triangle
together with $v_1$ and $v_2$, which contains $v_1$.   The
chordality of $G$ and the definition of $\mathcal{B}_G$
imply that $G_1$ is a tree, and thus we consider $G_1$ to be a
tree rooted at $v_1$.   The subgraph $G_2$ is analogously
defined.   Also analogously, the graph $G - \{ v_1, v_2 \}$
is a forest.   Let $G_v$ be the connected component
containing $v$ in $G - \{ v_1, v_2 \}$, and consider it as
a tree rooted at $v$.   Since there is a triangle in $G$
other than $A$, and $G$ is $F_1$-free, then $G_{v_0}$ (and
$G_v$ for any vertex $v$ forming a triangle together with
$v_1$ and $v_2$) has height at most $1$.

{\bf Case 2.1:} As a first subcase, suppose that $G_{v_0}$
has height $1$.   Since $G$ is $F_2$-free, then $G_1$ or
$G_2$ must have height $0$.   Assume without loss of
generality that $G_2$ has height $0$.   Since $G$ is
$F_1$-free, then $G_1$ has height at most $2$.   See the
center picture in Figure \ref{fig:MainTheorem}.

Thus, colour $v_1$ with colour $3$, $v_2$ with colour $1$,
every vertex $v$ forming a triangle together with $v_1$ and
$v_2$ with colour $2$, and every vertex at depth $1$ in
$G_v$ for every such $v$ with colour $2$.   Also, colour
every vertex in $G_1$ at depth $1$ with colour $2$, and
every at depth $2$ (if any) with colour $1$.   By the
previous analysis of the structure in this case, we have
that this colouring induces an $M_1$-partition of $G$.

{\bf Case 2.2:}   For every vertex $v$ forming a triangle
with $v_1$ and $v_2$, the rooted tree $G_v$ has height $0$.
Since $G$ is $F_3$-free, we assume without loss of generality
that $G_2$ has height $1$, and thus, $G_1$ might have height
$2$ (see the rightmost graph in Figure \ref{fig:MainTheorem}).

Hence, we colour $v_1$ with colour $3$, $v_2$ with colour $2$,
every vertex forming a triangle with $v_1$ and $v_2$ with
colour $1$, every neighbour of $v_2$ other than $v_2$ with
colour $1$, the vertices at depth $1$ in $G_1$ with colour $2$
and the vertices at depth $2$ in $G_1$ with colour $1$.

{\bf Case 3:}   Finally, suppose that $\mathcal{B}_G = \{ v_0 \}$.
As in the previous case, there must be at least two different
triangles in $G$.   We will first show that the eccentricity of
$v_0$ in $G$ is $2$.   To reach a contradiction, suppose that
$v_3$ is a vertex at distance $3$ from $v_0$, and let $(v_0, v_1
v_2, v_3)$ be a path in $G$ realizing this distance.   Since
$v_2$ and $v_3$ are not adjacent to $v_0$, then they are not in
any triangle of $G$, and thus, they belong to no cycle of $G$.
Since $v_1 \notin \mathcal{B}_G$, then there must be a triangle
$A$ in $G-v_1$ (which as mentioned, uses neither $v_2$ nor $v_3$).
Additionally, the edge $v_2 v_3$ is not adjacent to $A$, since
otherwise, as $A$ uses $v_0$, there would be a cycle in $G$ using
$v_2$ or $v_3$.   But now, the vertices of $A$ together with
$v_2$ and $v_3$ induce a copy of $F_1$, a contradiction.   Thus,
the eccentricity of $v_0$ is at most $2$.

Let $L_i$ be the set of all vertices at distance $i$ from $v_0$,
with $i \in \{ 1, 2 \}$.   From the discussion in the previous
paragraph, we have that $V = \{ v_0 \} \cup L_1 \cup L_2$.
Notice that $L_2$ is independent, as otherwise an edge between
vertices of $L_2$ would necessarily belong to a cycle, and hence
to a triangle of $G$, which in turn would imply that its vertices
are adjacent to $v_0$, contradicting the definition of $L_2$.
So, the neighbourhood of every vertex in $L_2$ is contained in
$L_1$.   Moreover, since vertices of $L_2$ do not belong to any
cycle of $G$, then they have degree $1$.   Again, since $v_0$ is
in every triangle of $G$, then the induced subgraph $G[L_1]$ is a
forest, and hence each of its connected components is uniquely
$2$-coloureable.

Suppose that there are vertices $v_1, v_2 \in L_1$ and vertices
$v_3, v_4 \in L_2$ such that $v_1 v_2, v_1 v_3$, and $v_2 v_4$
are edges in $G$.   See Figure \ref{fig:MainTheoremCase3} for
the configuration of this case.

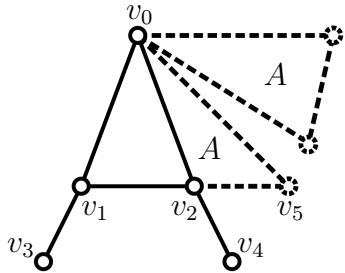
\begin{figure}
\centering
	\includegraphics[height=5cm]{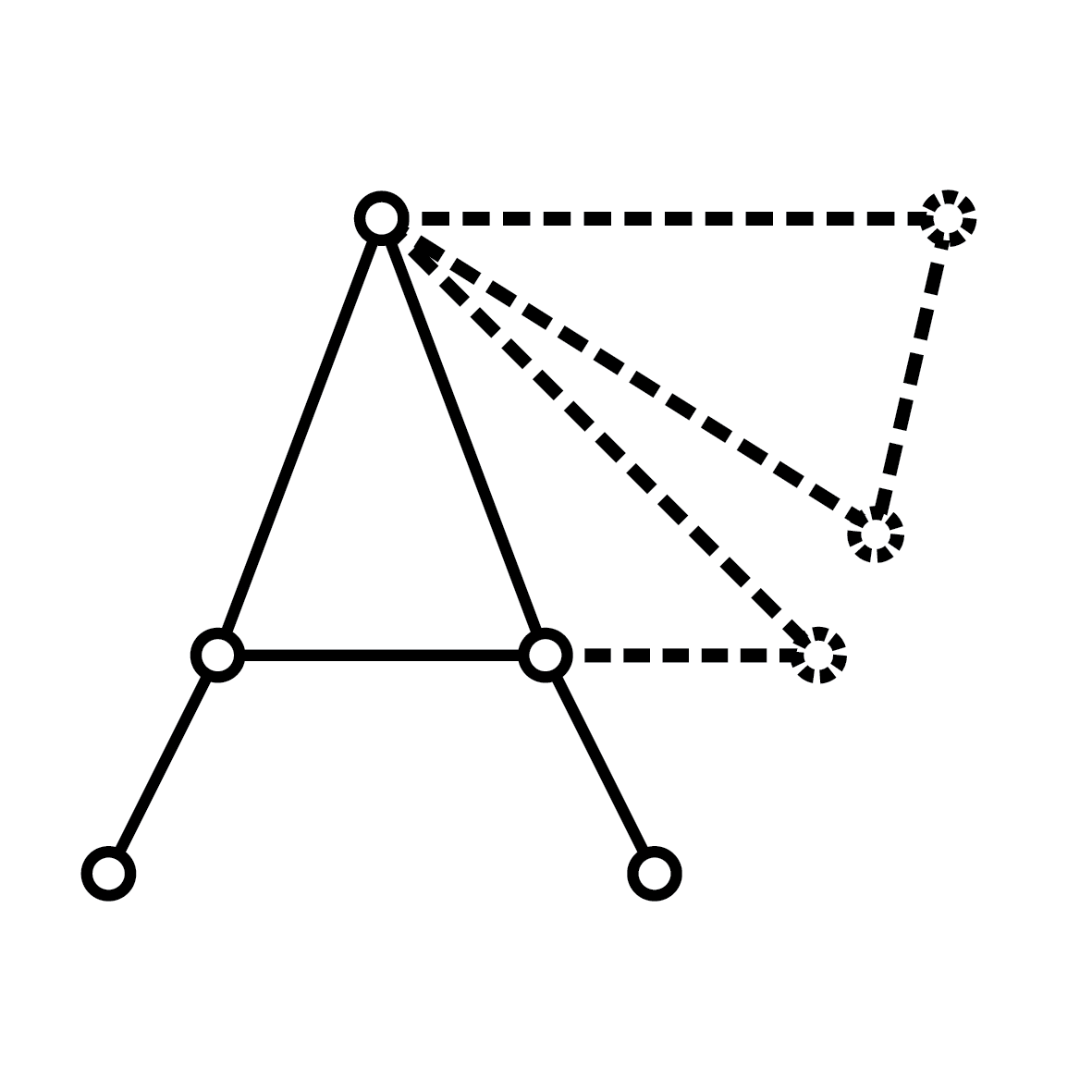}
	\put(-97,120){$v_0$}
	\put(-114,46){$v_1$}
	\put(-80,46){$v_2$}
	\put(-142,32){$v_3$}
	\put(-55,32){$v_4$}
	\put(-40,46){$v_5$}
	\put(-70,67){$A$}
	\put(-45,95){$A$}
\caption{Case 3 in Theorem \ref{thm:MainTheorem}.}
\label{fig:MainTheoremCase3}
\end{figure}

As $v_1$ is not in $\mathcal{B}_G$, there is triangle
$A$ in $G-v_1$.   If the two vertices in $A$ other
than $v_0$ are different from $v_2$, then at least
one of them is non-adjacent to $v_2$ (recall that
$G[L_1]$ is acyclic), and hence, this vertex together
with $v_0, v_1, v_2, v_3$, and $v_4$ induce a copy of
$F_2$.   Thus, one of the vertices of $A$ is $v_2$;
let us call the other vertex $v_5$ (see Figure
\ref{fig:MainTheoremCase3}).

Repeating the above procedure for $v_2$, there must
exist another triangle $B$ not containing $v_2$,
using $v_1$, and another vertex $v_6$, different
from $v_0$.  Recall that since $G[L_1]$ is acyclic,
then $v_5 \ne v_6$ and $v_1 v_5, v_2 v_6 \notin E_G$.
Thus, $\{ v_0, v_1, v_2, v_3, v_4, v_5, v_6 \}$ induces
a copy of $F_4$, a contradiction.

Hence, if vertices $v_1, v_2 \in L_1$ have neighbours in
$L_2$, then they cannot be adjacent.   Moreover, if
$v_1$ and $v_2$ have neighbours in $L_2$, as $G$ is
$\mathcal{F}_1$-free, then the distance between $v_1$
and $v_2$ in $G-v_0$ is even.   Thus, in every connected
component in $G[L_1]$, the vertices with neighbours in
$L_2$ are in the same part of the bipartition.

It follows from the previous description of the structure
of $G$ that the following colouring induces an
$M_1$-partition of $G$.   Colour $v_0$ with colour $3$
and every element of $L_2$ with colour $1$.   For every
connected component of $G[L_1]$, if it has any adjacencies
with $L_2$, (properly) $2$-colour it in such a way that
each vertex with neighbours in $L_2$ receives colour $2$;
else, $2$-colour it in any way.

Since the cases are exhaustive, the desired result follows.
\end{proof}

\section{Conclusions and further work}
\label{sec:conc}

We have provided a complete list of minimal obstructions
for one of the only two $3 \times 3$ matrices which have
inifinitely many chordal minimal obstructions.   This
nearly completes our knowledge on the matrix partition
problem for all $3 \times 3$ matrices on chordal graphs.
In \cite{federTCS349}, patterns $M$ with NP-complete
$M$-partition problems for chordal graphs are constructed,
but they are rather large (close to thirty rows and columns).
Hopefully, a complete understanding of the matrix partition
problem for small patterns will lead us to find the smallest
pattern $M$ having an NP-complete $M$-partition problem.

As the reader may notice, the exact list of chordal minimal
obstructions is still missing for $M_2$.  This list would
complete the analysis of all $3 \times 3$ patterns.
Fortunately, a complete list has been already obtained
as part of the first author's Ph.D. thesis, and a follow up
article including this result is in preparation.  An
interesting situation arises for $M_2$, there are two
different infinite families of chordal minimal obstructions
for the $M_2$-partition problem.

\section*{Acknwoledgments}

We would like to thank Sebasti\'an Gonz\'alez Hermosillo
de la Maza for pointing out Lemma \ref{lem:Bipart} to us.


\begin{thebibliography}{10}

\bibitem{bondy2008}
	A. Bondy and U.S.R. Murty.
	``Graph Theory''.
	Springer-Verlag, 2008.

\bibitem{damaschkeJGT14}
	P.~Damaschke,
	Induced subgraphs and well-quasi-ordering,
	Journal of Graph Theory 14 (1990) 427--435.

\bibitem{ekimDAM156}
	T.~Ekim, P.~Hell, J.~Stacho and D.~de~Werra,
	Polarity of chordal graphs,
	Discrete Applied Mathematics 156(13) (2008) 2469--2479.

\bibitem{federGTP2006}
	T.~Feder, P.~Hell and W.~Hochst\"attler,
	Generalized colourings of cographs, in
	Graph Theory in Paris, Birkh\"auser Verlag 2006, pp. 149--167.

\bibitem{federTCS349}
	T.~Feder, P.~Hell, S.~Klein, L.~T.~Nogueira and F.~Protti,
	List matrix partitions of chordal graphs,
	Theoretical Computer Science 349 (2005) 52--66.

\bibitem{federDM313}
	T.~Feder, P.~Hell and S.~N.~Rizi,
	Obstructions to partitions of chordal graphs,
	Discrete Mathematics 313 (2013) 1861--1871.

\bibitem{federDAM166}
	T.~Feder, P.~Hell and O.~Shklarsky,
	Matrix partitions of split graphs,
	Discrete Applied Mathematics 166 (2014) 91--96.

\bibitem{hellEJC35}
	P.~Hell,
	Graph partitions with prescribed patterns,
	European Journal of Combinatorics 35 (2014) 335--353.

\bibitem{hellDAM141}
	P.~Hell, S.~Klein, L.~T.~Nogueira and F.~Protti,
	Partitioning chordal graphs into independent sets and cliques,
	Discrete Applied Mathematics 141(1--3) (2004) 185--194.

\bibitem{hellDM338}
	P.~Hell and P.-L.~Yen,
	Join colourings of chordal graphs,
	Discrete Mathematics 338(12) (2015) 2453--2461.


\end{thebibliography}
\end{document}